\documentclass[a4paper, 11pt, twoside]{article}
\usepackage{amsmath,amsthm}
\usepackage{amssymb,latexsym}
\usepackage{mathrsfs}
\usepackage{enumerate}
\usepackage[colorlinks,citecolor=blue,dvipdfm]{hyperref}
\usepackage[numbers,sort&compress]{natbib}
\usepackage{indentfirst}

\DeclareMathAlphabet{\mathpzc}{OT1}{pzc}{m}{it}

\newtheorem{theorem}{Theorem}[section]

\newtheorem{lemma}{Lemma}[section]
\newtheorem{proposition}{Proposition}[section]
\newtheorem{definition}{Definition}[section]
\newtheorem{remark}{Remark}[section]

\theoremstyle{definition} \theoremstyle{remark}
\numberwithin{equation}{section}
\allowdisplaybreaks

\begin{document}
	\markboth{X. Huang, L. Peng, Y. Zhou}{Non-local evolution equations}
	\date{}
	\baselineskip 0.22in
	\title{\bf Non-local evolution equations with L\'{e}vy diffusion: Well-posedness and limiting behavior}
	
	\author{Xi Huang$^{1}$, Li Peng$^{1,2}$, Yong Zhou$^{1,3}$\thanks{\footnotesize {Corresponding author.}} \\[1.8mm]
		\footnotesize  {$^1$Faculty of Mathematics and Computational Science, Xiangtan University}\\
		\footnotesize  {Hunan 411105, China}\\
		\footnotesize  {$^2$Hunan Key Laboratory for Computation and Simulation in Science and Engineering}\\
		\footnotesize  {Xiangtan University, Hunan 411105, China}\\
		\footnotesize  {$^3$Faculty of Information Technology, Macau University of Science and Technology}\\
		\footnotesize  {Macau 999078, China}
	}
	
	\maketitle
	
	\begin{abstract}
		In this note we focus our attention on a class of nonlocal-in-time evolution equations with L\'{e}vy diffusion, they arise as models of unidirectional viscoelastic fluid flow and physical phenomena with memory effect.
		We first consider the existence of the classical solution to a nonlocal linear evolution problem under conditions on the involved memory kernels which allows complete positivity. Then we investigate the limit of this model to a generalized Rayleigh-Stokes equation, as the index of L\'{e}vy diffusion gets concentrated near two, we prove that the solution of nonlocal-in-time problem with L\'{e}vy diffusion uniformly converges to that of the generalized Rayleigh-Stokes equation and reveal the convergence rate.
		Finally, the existence and limiting behavior of the mild solution to a nonlocal evolution problem with nonlinearity are established. The proofs are based on subordination principle and relaxation function theory.\\[2mm]
	\end{abstract}
	
	\section{Introduction}
	This paper is concerned with  nonlocal evolution problem in the whole space $\mathbb R^N$ with $N\geq 1$:
	\begin{align}\label{eq 1.1}
		\begin{cases}
			\partial_{t}v + (-\Delta)^{\frac{\alpha}{2}}v + \gamma \partial_{t}(m \ast (-\Delta)^{\frac{\alpha}{2}} v) =g, \quad t>0, \ x \in \mathbb R^N, \\
			v|_{t=0}=v_{0}(x), \quad  x \in \mathbb R^N,
		\end{cases}
	\end{align}
	where $\gamma>0$, $\alpha \in (1,2]$ and the nonlocal operator $(-\Delta)^{\frac{\alpha}{2}}$ denotes the L\'{e}vy diffusion. $g$ represents the source linear term $g(t,x)$ or nonlinear source $g(v)$. The function $m = m(t)$ is a memory kernel and $\ast$ represents the convolution in time:	
	\begin{align*}
		(m \ast v)(t) = \int_{0}^{t} m(t-\tau)v(\tau)\, \mathrm{d}\tau, \quad t>0.
	\end{align*}
	
	The motivation for our research on problem \eqref{eq 1.1} stems from the similar challenges encountered recently in the modeling of unidirectional viscoelastic fluid flow, anomalous to classical diffusion and physical phenomena with memory effect.
	
	We write $m_{\gamma}(t):= 1+ \gamma m(t), \ t>0$.  The memory kernel $m$ satisfies the condition:
	\begin{itemize}
		\item[$(\mathcal{H})$] $m\in L_{loc}^1(\mathbb R_+)$, and there exist a nonnegative and nonincreasing function $k_{\gamma}\in L_{loc}^1(\mathbb R_+)$ and $\beta_\gamma \geq 0$ such that
		\begin{align}\label{1.2}
			\beta_\gamma m_{\gamma}(t) + (k_{\gamma}\ast m_{\gamma})(t) = 1, \quad t>0.
		\end{align}
	\end{itemize}
	It follows from \cite[Theorem 2.2]{P. Clement 81} that the function $m_\gamma$ is completely positive.

	When $\alpha = 2$, we obtain a generalized Rayleigh-Stokes equation
	\begin{align}\label{1.5}
		\partial_{t}v -\Delta
		v - \gamma \partial_{t}(m \ast \Delta v) =g.
	\end{align}
	There exists a vast literature regarding the equation \eqref{1.5} due to the wide application.
	An extensive use of the eigenvalue expansion method was advocated in many results, this allows analyzing the behavior of solutions for the equations involving the Laplacian operator. For example,
	Tuan et al. \cite{P. T. Tuan} proved the H\"{o}lder continuity and regularity of solutions to the backward problem of the equation \eqref{1.5} with weakly-valued nonlinearity. Later, a more general form than the equation \eqref{1.5} was considered in \cite{T. D. Ke}, and global existence and stability of mild solutions were obtained. In our recent works \cite{L. Peng2024}, well-posedness and asymptotic profile analysis of the solutions to a generalized Rayleigh-Stokes equation were obtained in Besov-Morrey space, the proof was based on Bernsein functions and the theory of Karamata-Feller Tauberian.
	For many further results on
	solvability, regularity and stability in this direction, we can refer to the papers \cite{T. D. Ke 22, D. Van Loi, T. Van Tuan}.
	
	Different forms of the kernel $m$ be covered by our investigations.
	The choice $m(t) = \frac{t^{-\beta}}{\Gamma(1-\beta)}, \ \beta \in (0,1)$, leads to a Rayleigh-Stokes equation with Riemann fractional derivatives, which was introduced in
	\cite{Y. Zhou, Y.Zhou 1}. In this case, it was derived in \cite{F. Shen, M. Khan} for a heated edge or viscoelastic fluid, and  also called a fractional Rayleigh-Stokes equation. 
	In \cite{J. N. Wang}, it is demonstrated the well-posedness and blow up of $\varepsilon-$regular solution.
	Very recently, the well-posedness and asymptotic profile of global solutions in
	Besov spaces can be found  in our paper \cite{L. Peng} from the perspective of resolvent operators.
	
	For the case where $m\in C^1(\mathbb R_+)$, the equation \eqref{1.5} can  be generally written in a more specific form
	\begin{align*}
		\partial_{t}v -(1+\gamma m(0))\Delta v - \gamma m' \ast \Delta v =g.
	\end{align*}
	This model provides a faithful description of heat conduction with memory effect, see\cite{M. Fabrizio, J. R. Cannon, R. Nachlinger, L. Von Wolfersdorf}. In addition, the abstract form of above equation has been studied extensively, see \cite{P. Cannarsa, N. M. Dien, E. Bazhlekova 152}.
	
	When $\alpha \in (1,2)$, the operator $(-\Delta)^{\frac{\alpha}{2}}$ is also known as the fractional Laplacian.
	In the case $m \equiv a_0$, $a_0 \geq 0$ is a positive constant, the equation \eqref{eq 1.1} generalizes the well-known space-fractional diffusion equation
	\begin{align*}
		\partial_t v + (1+a_0\gamma) (-\Delta)^{\frac{\alpha}{2}}v = g.
	\end{align*}
	This model is connected with macroscopic description of transport resulting in superdiffusion phenomenont and there exists a substantial body of literature on the topic.
	Zhai \cite{Z. Zhai} developed several Strichartz type estimates in Besov spaces and Sobolev spaces. Later, Bonforte et al. \cite{M. Bonforte} addressed the existence, uniqueness and regularity of weak solutions to fractional heat equations with the restriction that the initial value is a nonnegative and locally bounded Radon measure.
	Recently, Jarr\'{i}n and Loacham\'{i}n \cite{O. Jarrin} considered global existence of classical solutions to fractional heat equations with a polynomial non-linear term, and derived convergence behavior of solutions from non-local to local case. Notably, Van Dac et al. \cite{N. Van Dac} gave some sufficient conditions ensuring the existence, stability, and H\"{o}lder regularity of solutions under setting in Hilbert scales regarding the case of the kernel $m(t)= \frac{t^{-\beta}}{\Gamma(1-\beta)}$.
	
	Let us now go back to the study on the model \eqref{eq 1.1}.
	Even if we can use the subordinate principle to present the solution operator of the problem \eqref{eq 1.1} as
	\begin{align*}
		S_{\gamma,\alpha}(t)=-\int_0^\infty e^{-\rho(-\Delta)^{\frac{\alpha}{2}}} w_\gamma(t,\mathrm{d}\rho),
	\end{align*}
	where $w_\gamma(t,\tau)$ is called the propagation function associated with the completely positive kernel $m_\gamma$, see \cite[Proposition 4.9]{J. Pruss}. At present, it is generally difficult to use the known properties of the propagation function for further study on the solution operators $S_{\gamma,\alpha}$ apart from mild solutions. However, in our work, we succeeded in the strong continuity, differentiability and spatial regularity of the fundamental solution $Q_{\gamma,\alpha}(t,x)$ in the Sobolev space by the relaxation functions, which removes the obstacle to obtaining classical solutions. Moreover, asymptotic estimates of the fundamental solution $Q_{\gamma,\alpha}$ with respect to the diffusion parameter $\alpha$ are given, this also implies that the relaxation function is a powerful tool to deal with the dependence on parameters of solutions to abnormal diffusion problems.
	Inspired by existing works, in this paper, we consider the existence of classical/mild  solutions to nonlocal evolution problem
	\begin{align}\label{neweq 1.1}
		\begin{cases}
			\partial_{t}v_\alpha + (-\Delta)^{\frac{\alpha}{2}}
			v_\alpha + \gamma \partial_{t}(m \ast (-\Delta)^{\frac{\alpha}{2}} v_\alpha) =g, \quad t>0, \ x \in \mathbb R^N, \\
			v_{\alpha}|_{t=0}=v_{0,\alpha}(x), \quad  x \in \mathbb R^N,
		\end{cases}
	\end{align}
	where $v_{0,\alpha}$ and $v_\alpha$ represent the initial condition and the solution, respectively, corresponding to the parameter $\alpha$.
	The second results considers the limit of solutions to the problem \eqref{neweq 1.1} when $\alpha\to 2^-$. The limit solves the following equation
	\begin{align}\label{new1.5}
		\partial_{t}v -\Delta
		v - \gamma \partial_{t}(m \ast \Delta v) =g,~~~v|_{t=0}=v_{0,2}(x),
	\end{align}
	interpreted in terms of mild formulation under the assumption that $v_{0,\alpha}\to v_{0,2}$ in space $H^s$.
	
	The rest of this paper is organized as follows. In Section 2, we list some notions, relaxation function and fundamental solutions that are needed in the sequel, and then we show that the fundamental solution $Q_{\gamma,\alpha}$ is substantively a probability density function of a random process.  In Section 3, we obtain the smooth properties of relaxation function, then estimates and continuity of fundamental solutions are given in Sobolev spaces. In Section 4, we establish the existence and convergence of classical solutions to the linear problem. Finally, the existence and convergence of mild solutions to the nonlinear problem are derived.
	
	\section{Preliminaries}
	
	\subsection{Notations}
	We first introduce some notation used throughout the paper. Let $f_1 \in \mathcal{S}(\mathbb R^N)$, $\tilde f_1$ denotes the Fourier transform of, and $\mathcal F^{-1}$ denotes the inverse Fourier transform. For $f_2:(0, \infty) \to \mathbb R$, $\widehat{f_2}$ denotes Laplace transform and $\mathcal L^{-1}(f_2)(\lambda)$ denotes the inverse Laplace transform. Note $H^s=:H^s(\mathbb R^N)$ and ${\dot H}^s=:{\dot H}^s(\mathbb R^N)$ as non-homogeneous and homogeneous fractional Sobolev spaces, respectively.
	
	Let $1<\alpha \leq 2$. The operator $(-\Delta)^{\frac{\alpha}{2}}:\mathcal{S}(\mathbb R^N) \to \mathcal{S}(\mathbb R^N)$ is given by means of the Fourier multiplier
	\begin{align*}
		(-\Delta)^{\frac{\alpha}{2}} f_1 = \mathcal F^{-1}(|\xi|^\alpha \tilde{f_1}(\xi))(x).
	\end{align*}
	It is well-known that $(-\Delta)^{\frac{\alpha}{2}}$ generates a $C_0$-semigroup $\{e^{-t(-\Delta)^{\frac{\alpha}{2}}}\}_{t\geq0}$ on $L^2(\mathbb R^N)$ for $\alpha\in(1, 2]$, which can be represented as
	\begin{align*}
		e^{-t(-\Delta)^{\frac{\alpha}{2}}}v=[H_\alpha(t,\cdot)\star v](x), \quad t>0,~x\in\mathbb R^N,
	\end{align*}
	where $\star$ is the convolution in $\mathbb R^N$ and the function $H_\alpha(t,x)$ is the fundamental solution of the homogeneous space-fractional heat equation
	\begin{align*}
		\partial_tH_\alpha + (-\Delta)^{\frac{\alpha}{2}} H_\alpha = 0, \quad t>0.
	\end{align*}
	The Fourier transform of $H_\alpha(t,x)$ is $\tilde{H_\alpha}(t,\xi) = e^{-t|\xi|^\alpha}$.
	Moreover, $H_\alpha(t,x)$ is positive and smooth. Specially, $H_2(t,x)$ is Guassian heat kernel for $\alpha=2$.
	
	\subsection{Relaxation function}	
	For $\nu \geq 0$, we consider the Volterra integral equation
	\begin{align}
		&s_\gamma(t,\nu) + \nu m_\gamma \ast s_\gamma(\cdot,\nu)(t)=1,\quad t \geq 0,\label{2.1} \\
		&r_\gamma(t,\nu) + \nu m_\gamma \ast r_\gamma(\cdot,\nu)(t)=m_\gamma(t),\quad t>0. \label{2.2}
	\end{align}
	The condition $(\mathcal H)$ shows that $m_\gamma$ is complete positive, which also ensures the nonnegativity of $m_\gamma$, see \cite[Proposition 2.1]{P. Clement 81}.
	As mentioned in \cite[Definition 1.1]{P. Clement 81} and \cite[Proposition 4.5]{J. Pruss},
	the complete positivity of $m_\gamma$ insures the nonnegativity of $s_\gamma(t,\nu)$ and $r_\gamma(t,\nu)$, here $s_\gamma(t,\nu)$ and $r_\gamma(t,\nu)$ are the unique solution of \eqref{2.1} and \eqref{2.2}, respectively.
	Moreover, several properties of the relaxation functions $s_\gamma(t, \nu)$ and $r_\gamma(t,\nu)$ are collected.
	\begin{proposition}\label{prop 2.1}
		Assume $m\in(\mathcal{H})$ and $\nu \geq 0$. Then we have the following results.
		\begin{itemize}
			\item[(i)] Let $f \in L_{loc}^1(\mathbb R_+)$ and $t>0$. The Volterra integral equation
			$b_\gamma(t, \nu) + \nu m_{\gamma} \ast b_\gamma(t, \nu) = f(t) $
			admits a unique solution
			$b_\gamma(t,\nu) = f(t) - \nu r_\gamma(\cdot, \nu) \ast f(t)$. If $f \in C(\mathbb R_+)$, then $b_\gamma(\cdot) \in C(\mathbb R_+)$ and the previous form of $b_\gamma(t)$ also holds for $t=0$.
			\item[(ii)] The function $s_\gamma(t,\nu)$ has the following representation
			\begin{align*}
				s_\gamma(t, \nu) = 1 - \nu \int_0^t r_\gamma(\tau, \nu)\, \mathrm{d}\tau, \quad t \geq 0,
			\end{align*}
			which implies that $s_\gamma(\cdot,\nu)$  is continuous on $\mathbb R_+$ and differentiable on $(0,\infty)$.
			\item[(iii)] For $t >0$, it holds
			\begin{align*}
				0\leq& r_\gamma(t,\nu) \leq 1+\gamma m(t),\\
				\frac{1+\beta_\gamma \nu[k_\gamma(t)]^{-1}r_\gamma(t,\nu)}{1+\nu [k_\gamma(t)]^{-1}} \leq&  s_\gamma(t,\nu) \leq \frac{1}{1+\nu[t+\gamma(1\ast m)(t)]}
				.
			\end{align*}
			Moreover if $m$ is nonincreasing, then
			\begin{align}\label{new(2.1)}
				r_\gamma(t,\nu) \leq \frac{1+\gamma m(t)}{1+\nu[t+\gamma(1\ast m)(t)]}.
			\end{align}
			
			\item[(iv)] For $t >0$, the functions $s_\gamma(t,\nu)$ and $r_\gamma(t,\nu)$ can be expressed separately as
			\begin{align*}
				s_\gamma(t, \nu)=-\int_{0}^{\infty} e^{-\rho \nu} w_\gamma(t, \mathrm{d}\rho) \quad {\it and}~~
				r_\gamma(t, \nu)=\int_{0}^{\infty} e^{-\rho \nu} \eta_\gamma(t, \mathrm{d}\rho),
			\end{align*}
			where $w_\gamma(t,\rho)$ is the propagation function associated with the kernel $m_{\gamma}$, and $-w_\gamma(t, \mathrm{d} \rho)$ and $\eta_\gamma(t, \mathrm{d}\rho)$ are positive finite measures satisfying $$-\int_0^\infty w_\gamma(t,\mathrm{d}\rho)=1, ~~~~ \int_0^\infty \eta_\gamma(t,\mathrm{d}\rho)=m_{\gamma}(t),  \quad t>0.$$
		\end{itemize}
	\end{proposition}
	\begin{proof}
		The justification for (i) and (iv) can be founded in
		\cite[Chapter 2, Theorem 3.5]{G. Gripenberg} and \cite[Theorem 1]{P. Clement 79}, respectively. The statement (ii) follows by the statement (i) for $f\equiv1$.
		
		For (iii), the first inequality is obvious for the nonnegativity of $m$ and $r_\gamma(t,\nu)$.
		The right-handed side of the second inequality stems from the proof of
		\cite[Proposition 2.1]{P. Clement 81}. It remains to prove the left-handed side of the second inequality. Using \eqref{1.2} and \eqref{2.1}, we have
		\begin{align*}
			\beta_\gamma r_\gamma(t,\nu) + (k_\gamma \ast r_\gamma(\cdot,\nu))(t)+ \nu \left(m_\gamma \ast \left(\beta_\gamma r_\gamma(\cdot,\nu)+ k_\gamma \ast r_\gamma(\cdot,\nu) \right)\right)(t)=1,\quad t \geq 0,
		\end{align*}
		this together with the uniqueness of solutions to \eqref{2.2} shows
		\begin{align*}
			s_\gamma(t,\nu) = \beta_\gamma r_\gamma(t,\nu) + (k_\gamma \ast r_\gamma(\cdot,\nu))(t), \quad t\geq 0.
		\end{align*}
		Since $\partial_t s_\gamma(t,\nu) = -\nu r_\gamma(t,\nu) <0$ for a.e. $t \in (0,\infty)$		(due to (ii)) and the function $k_\gamma$ is nonincreasing, we have
		\begin{align*}
			\nu s_\gamma(t,\nu) &= \beta_\gamma \nu r_\gamma(t,\nu) -\int_0^t k_\gamma(t-\tau)\partial_\tau s_\gamma(\tau,\nu)\, \mathrm{d}\tau\\
			&\geq \beta_\gamma \nu r_\gamma(t,\nu)-k_\gamma(t) \int_0^t \partial_\tau s_\gamma(\tau,\nu)\, \mathrm{d}\tau\\
			&= \beta_\gamma \nu r_\gamma(t,\nu)+k_\gamma(t)(1 - s_\gamma(t,\nu)).
		\end{align*}
		that for $t\geq 0$. Therefore, the left-handed side of the second inequality holds.
		Finally, the inequality \eqref{new(2.1)} follows by the similar proof of
		\cite[Lemma 5.4]{J.C. Pozo}.
	\end{proof}
	
	\subsection{Representation of fundamental solutions}	
	Consider the homogeneous case of the equation \eqref{eq 1.1}. Let the function $Q_{\gamma, \alpha}(t, x)$ be the fundamental solution to the following equations:
	\begin{align}\label{2.3}
		\begin{cases}
			\partial_{t}Q_{\gamma,\alpha} + (-\Delta)^{\frac{\alpha}{2}} Q_{\gamma,\alpha} + \gamma \partial_{t}(m \ast (-\Delta)^{\frac{\alpha}{2}} Q_{\gamma,\alpha}) =0, \quad t>0,\ x \in \mathbb R^N, \\
			Q_{\gamma,\alpha}|_{t=0}=\delta_0(x), \quad x \in \mathbb R^N,
		\end{cases}
	\end{align}
	where $\delta_0$ is the Dirac measure. Then $\tilde{Q}_{\gamma,\alpha}$ solves
	\begin{align*}
		\begin{cases}
			\partial_{t}\tilde{Q}_{\gamma,\alpha}(t,\xi) + |\xi|^{\alpha} \partial_t m_\gamma\ast \tilde{Q}_{\gamma,\alpha}(t, \xi) =0, \quad t>0, \ \xi \in \mathbb R^N, \\
			\tilde{Q}_{\gamma,\alpha}(0,\xi)=1, \quad \xi \in \mathbb R^N,
		\end{cases}
	\end{align*}
	which is equivalent to the following Volterra equation:
	\begin{align}\label{2.4}
		\tilde{Q}_{\gamma,\alpha}(t,\xi) + |\xi|^{\alpha} m_\gamma \ast \tilde{Q}_{\gamma,\alpha}(t, \xi) =1, \quad t > 0, \ \xi \in \mathbb R^N.
	\end{align}
	By Proposition \ref{prop 2.1}(v) we have
	\begin{align}\label{2.5}
		\tilde{Q}_{\gamma, \alpha}(t, \xi) = s_\gamma(t,|\xi|^\alpha)= -\int_{0}^{\infty} e^{-\rho |\xi|^\alpha} w_\gamma(t, \mathrm{d}\rho),
	\end{align}
	using this and the fact that $\tilde{H_\alpha}(t,\xi)=e^{-t|\xi|^\alpha}$, we derive
	\begin{align}\label{2.6}
		Q_{\gamma, \alpha}(t,x)  =-\int_0^\infty H_\alpha(\rho,x) w_\gamma(t,\mathrm{d}\rho), \quad t > 0, \ x \in \mathbb R^N.
	\end{align}
	
	On the other hand, we note that the equation \eqref{eq 1.1} is equivalent to the Volterra equation
	\begin{align*}
		v_{\alpha} + m_\gamma \ast (-\Delta)^{\frac{\alpha}{2}}v_{\alpha} = v_{0,\alpha}+1\ast g.
	\end{align*}
	From the variation of parameters formula for Volterra equation, the equation \eqref{eq 1.1} has a solution of the form
	\begin{align*}
		v_\alpha(t, x)= Q_{\gamma,\alpha}(t,\cdot) \star v_{0,\alpha} + \int_0^t Q_{\gamma,\alpha}(t-\tau, \cdot) \star g(\tau,v_\alpha) \mathrm{d}\tau.
	\end{align*}
	
	Now we define the function
	$G_{\gamma,\alpha}(t, x)=\int_{0}^{\infty} H_\alpha(\rho, x) \eta_\gamma(t, \mathrm{d}\rho)$ for $t>0, \ x \in \mathbb R^N$.
	It follows from Proposition \ref{2.1}(iv) that
	\begin{align}\label{2.7}
		\tilde G_{\gamma,\alpha}(t, \xi)=r_{\gamma}(t, |\xi|^\alpha)=\int_{0}^{\infty} e^{-\rho |\xi|^\alpha} \eta_\gamma(t, \mathrm{d}\rho), \quad t>0.
	\end{align}
	
	For the two functions $Q_{\gamma,\alpha}$ and $G_{\gamma,\alpha}$, the following assertions holds:
	\begin{proposition}\label{prop 2.2}
		Let $(\mathcal{H})$ hold. then for $\beta_\gamma$ and $k_\gamma$ given as in $(\mathcal{H})$,
		\begin{itemize}
			\item[(i)] the function $G_{\gamma,\alpha}(t,x)$ is a solution of the Volterra integral equation
			\begin{align*}
				\beta_\gamma G_{\gamma,\alpha}(t,x)+k_\gamma\ast G_{\gamma,\alpha}(\cdot,x)(t)= Q_{\gamma,\alpha}(t,x) \quad{\it for}~ t>0, \text{~a.e.~} x \in \mathbb R^N.
			\end{align*}
			
			\item[(ii)]$1 \ast G_{\gamma,\alpha}(\cdot, x)(t) =m_\gamma \ast Q_{\gamma,\alpha}(\cdot,x)(t) \quad{\it for}~ t>0, \text{~a.e.~} x \in \mathbb R^N.$
		\end{itemize}
	\end{proposition}
	\begin{proof}
		(i)	We take the Laplace transform of the equations \eqref{2.2} and \eqref{2.4} to obtain
		\begin{align}\label{2.8}
			\widehat {\tilde Q}_{\gamma,\alpha}(\lambda, \xi) = \frac{1}{\lambda+|\xi|^\alpha \lambda \widehat m_\gamma(\lambda) }, \quad
			\widehat {\tilde G}_{\gamma,\alpha}(\lambda, \xi) = \frac{\widehat m_\gamma(\lambda)}{1+|\xi|^\alpha  \widehat m_\gamma(\lambda) }, \quad Re \lambda>0, \ \xi \in \mathbb R^N.
		\end{align}
		In view of $m \in (\mathcal {H})$, one can see
		\begin{align*}
			\widehat m_\gamma(\lambda)\left(\beta+\widehat k_\gamma(\lambda)\right)=\frac{1}{\lambda}, \quad Re \lambda >0,
		\end{align*}
		This shows
		$$
		\widehat {\tilde Q}_{\gamma,\alpha}(\lambda, \xi)-\left(\beta+\widehat {k}_\gamma(\lambda)\right)\widehat {\tilde G}_{\gamma,\alpha}(\lambda, \xi)=0.
		$$
		Then $(i)$ follows by the inverse Laplace transform and inverse Fourier transform.
		
		(ii)  In view of
		$$\frac{1}{\lambda} \widehat {\tilde{G}}_{\gamma,\alpha}(\lambda, \xi)-\widehat m_\gamma(\lambda)\widehat {\tilde Q}_{\gamma,\alpha}(\lambda, \xi)=0,
		$$
		(ii) follows by the similar arguments.
	\end{proof}
	
	Next, from a measure perspective we show that the fundamental solution $Q_{\gamma,\alpha}$ is substantively the probability density function of a random process for arbitrary $N \in \mathbb N^+$.
	\begin{theorem}\label{the 2.1}
		Let $m$ satisfy $(\mathcal H)$ and $\alpha \in (1,2]$. Then the fundamental solution $Q_{\gamma,\alpha}(t,x)$ of the problem
		\eqref{eq 1.1} can be interpreted as the probability density function of a random process $Y_{\gamma,\alpha}(t)$. Moreover, mean square displacement of $Y_{\gamma,2}$ is
		\begin{align*}
			msd_{\gamma,2}(t)= 2 \left(t+\gamma(1 \ast m)(t)\right)~~{\it for}~\ t > 0.
		\end{align*}
	\end{theorem}
	\begin{proof}
		Since $H_\alpha(t,x)$ is positive on $(0, \infty) \times \mathbb R^N$ and $-w_\gamma(t, \mathrm{d}\tau)$ is a positive measure, it follows from \eqref{2.6} that $Q_{\gamma,\alpha}(t,x)\geq 0$  for $t>0$. Using \eqref{2.8} and Fubini theorem, we know
		\begin{align*}
			\int_{0}^{\infty} e^{-\lambda t} \int_{\mathbb R^N}Q_{\gamma,\alpha}(t, x)\, \mathrm{d}x \mathrm{d}t
			&= \int_{0}^{\infty} \int_{\mathbb R^N} e^{-\lambda t} Q_{\gamma,\alpha}(t, x)\, \mathrm{d}x \mathrm{d}t\\
			&= \int_{0}^{\infty} \int_{\mathbb R^N} \left. \left(e^{-ix \cdot \xi}\right)\right|_{\xi=0}  e^{-\lambda t} Q_{\gamma,\alpha}(t, x)\, \mathrm{d}x \mathrm{d}t\\
			&=\left. \left( \int_{0}^{\infty} \int_{\mathbb R^N} e^{-ix \cdot \xi}  e^{-\lambda t} Q_{\gamma,\alpha}(t, x)\, \mathrm{d}x \mathrm{d}t \right)\right|_{\xi=0} \\
			&=\widehat {\tilde Q}_{\gamma,\alpha}(\lambda, 0)=\frac{1}{\lambda}.
		\end{align*}
		The inverse Laplace transform yields
		$$
		\int_{\mathbb R^N}Q_{\gamma,\alpha}(t, x)\, \mathrm{d}x=1, \quad t>0,	
		$$
		it also proves $Q_{\gamma,\alpha}(t,x)$ is the probability density function $Y(t)$.
		
		For the Laplace transform of $msd_{\gamma,2}$ we have
		\begin{align*}
			\widehat{msd}_{\gamma,2}(\lambda)=
			&=\int_{0}^{\infty} e^{-\lambda t} \int_{\mathbb R^N}|x|^2Q_{\gamma,2}(t, x)\, \mathrm{d}x \mathrm{d}t
			= \int_{0}^{\infty} \int_{\mathbb R^N} e^{-\lambda t} |x|^2Q_{\gamma,2}(t, x)\, \mathrm{d}x \mathrm{d}t\\
			&= \int_{0}^{\infty} \int_{\mathbb R^N} \left. \left(-\Delta_\xi e^{-ix \cdot \xi}\right)\right|_{\xi=0}  e^{-\lambda t} Q_{\gamma,2}(t, x)\, \mathrm{d}x \mathrm{d}t\\
			&=\left. \left(-\Delta_\xi  \int_{0}^{\infty} \int_{\mathbb R^N} e^{-ix \cdot \xi}  e^{-\lambda t} Q_{\gamma,2}(t, x)\, \mathrm{d}x \mathrm{d}t \right)\right|_{\xi=0} \\
			&=-\Delta_\xi\left. \left(\widehat {\tilde Q}_{\gamma,2}(\lambda, \xi)\right) \right|_{\xi=0}=\frac{2 \left(1+\gamma \lambda \widehat m(\lambda)\right)}{\lambda^2}.
		\end{align*}
		Using the inverse Laplace transform again, the assertion follows.
	\end{proof}

	\section{Auxiliary results}
	In this section, we will give a number of statements characterizing smoothing properties of relaxation functions and solution operators.	
	\subsection{Properties of relaxation functions}	
	Now we state several in-depth properties of $s_\gamma(t,\nu)$ and $r_\gamma(t,\nu)$. First, we give a meticulous proof of the continuity of $r_\gamma(t,\nu)$ with respect to $t$.
	\begin{lemma}\label{newprop 2.2}
		Assume $m\in(\mathcal{H})$ and $\nu \geq 0$. Then the function $r_\gamma(\cdot,\nu)$ is continuous on $(0,\infty)$ for $\beta_\gamma>0$. If $m$ is nonincreasing and continuous, then  $r_\gamma(\cdot,\nu)$ is continuous on $(0,\infty)$ for $\beta_\gamma =0$.
	\end{lemma}
	\begin{proof}
		Let $\beta_\gamma>0$.  From \eqref{1.2} we have
		$$m_{\gamma}(t) + \frac{1}{\beta_\gamma}(k_{\gamma}\ast m_{\gamma})(t) = \frac{1}{\beta_\gamma}\quad {\rm for}~~t>0.$$
		It follows from Proposition \ref{prop 2.1} (i) that
		$m_\gamma(t) = \frac{1}{\beta_\gamma}[1 - (1 \ast a_\gamma)(t)]$ for $t\geq0$,
		where $a_\gamma \in L^1_{loc}(\mathbb R_+)$ is a unique solution of the following integral equation
		\begin{align*}
			a_\gamma(t)+\frac{1}{\beta_\gamma}(k_\gamma \ast a_\gamma)(t) =
			\frac{ k_\gamma(t)}{\beta_\gamma} \quad {\rm for}~~t>0.
		\end{align*}
It ensures the continuity of $m_\gamma(t)$ on $\mathbb R_+$. Moreover, $m_\gamma$ is also non-negative and bounded (due to \cite[Lemma 1.3]{J. J. Levin}). Using \cite[Chapter 3, Corollary 6.2]{G. Gripenberg} we can see that $(m \ast r_\gamma(t,\nu))(t)$ is continuous on $\mathbb R_+$. Then Proposition \ref{prop 2.1} (i) shows that $r_\gamma(\cdot,\nu)$ is also continuous on $\mathbb R_+$.
		
		Let $\beta_\gamma =0$. Fix $t \in (0,\infty)$, we choose $h \in (0,\frac{t}{2})$, then we have
		\begin{align*}
			&\quad\left|(m_\gamma \ast r_\gamma(\cdot,\nu))(t+h) - (m_\gamma \ast r_\gamma(\cdot, \nu)(t) \right|\\
			&\leq \int_0^t |m_\gamma(t+h-\tau) - m_\gamma(t-\tau)|r_\gamma(\tau,\nu)\,  \mathrm{d}\tau
			+ \int_t^{t+h} m_\gamma(t+h-\tau) r_\gamma(\tau,\nu)\, \mathrm{d}\tau\\
			&:= I_1(t,h)+I_2(t,h).
		\end{align*}
		Since $m$ is nonincreasing and continuous on $(0,\infty)$, it follows that
		\begin{align*}
			|m_\gamma(t+h-\tau) - m_\gamma(t-\tau)|r_\gamma(\tau,\nu) \leq 2 m(t-\tau)r_\gamma(\tau, \nu),
		\end{align*}
		and $\lim_{h \to 0^+} m_\gamma(t+h-\tau) = m_\gamma(t-\tau)$ for $\tau \in (0,t)$. Then by the Lebesgue's dominated convergence theorem we have $\lim_{h \to 0^+} I_1(t,h) \to 0$. For $I_2(t,h)$, using Proposition \ref{2.1} (iii), one can see that
		\begin{align*}
			0 \leq I_2(t,h) \leq \int_t^{t+h} m_\gamma(t+h-\tau)m_\gamma(\tau)\, \mathrm{d}\tau
			\leq \sup_{\zeta \in [t, \frac{3}{2}t]} m_\gamma(\zeta) \int_0^{h} m_\gamma(\tau)\, \mathrm{d}\tau \to 0 \text{ ~as~ } h \to 0^+.
		\end{align*}
		This ensures the right-continuity of $(m \ast r_\gamma(\cdot, \nu))(t)$.
		
		Now we prove the left-continuity of $(m \ast r_\gamma(\cdot, \nu))(t)$. For any $\epsilon \in (0, \frac{t}{2})$, we have
		\begin{align*}
			&\quad\left|(m \ast r_\gamma(\cdot,\nu))(t-h) - (m_\gamma \ast r_\gamma(\cdot, \nu)(t) \right|\\
			&\leq \int_{\epsilon}^{t-h} \left(m_\gamma(t-h-\tau) - m_\gamma(t-\tau)\right)m_\gamma(\tau)\,  \mathrm{d}\tau\\
			&\quad +\int_0^{\epsilon} \left(m_\gamma(t-h-\tau) - m_\gamma(t-\tau)\right)m_\gamma(\tau)\,  \mathrm{d}\tau
			+ \int_{t-h}^t m_\gamma(t-h-\tau) m_\gamma(\tau)\, \mathrm{d}\tau\\
			&:= I_{3}(t,h) + I_{4}(t,h)+I_5(t,h).
		\end{align*}
		From the nonincreasity of $m$ and Proposition \ref{prop 2.1} (iv), it is trivial to check that
		\begin{align*}
			&0\leq I_{3}(t,h) = m_\gamma(\epsilon)\left[(1\ast m_\gamma)(t-h-\epsilon) - (1\ast m_\gamma)(t-\epsilon) +(1\ast m_\gamma)(h) \right] \to 0 \text{ ~as~ } h \to 0^+,\\
			&0 \leq I_5(t,h) \leq \sup_{\zeta \in [\frac{t}{2}, t]} m_\gamma(\zeta) \int_0^{h} m_\gamma(\tau)\, \mathrm{d}\tau \to 0 \text{ ~as~ } h \to 0^+.
		\end{align*}
		Applying the nonincreasity and continuity of $m$ on $(0,\infty)$, we get
		\begin{align*}
			\left(m_\gamma(t-h-\tau) - m_\gamma(t-\tau)\right)m_\gamma(\tau) \leq 2 m\left(\frac{t}{2}-\epsilon\right)m_\gamma(\tau),
		\end{align*}
		and $\lim_{h \to 0^+} m_\gamma(t-h-\tau) = m_\gamma(t-\tau)$ for $\tau \in (0,\epsilon)$. Then it yields $\lim_{h \to 0^+} I_{4}(t,h) \to 0$ by Lebesgue's dominated convergence theorem again. It ensures the left-continuity of $(m \ast r_\gamma(\cdot, \nu))(t)$. This together with \eqref{2.2} shows that $r_\gamma(\cdot,\nu)$ is continuous on $(0,\infty)$.
	\end{proof}
	
	In the following lemmas, we give a different proof of the differentiability of $\nu \mapsto s_\gamma(t,\nu)$ and $\nu \mapsto s_\gamma(t,\nu)$ from \cite[Lemma 5.1]{J. Kemppainen} and \cite[Lemma 5.3]{J.C. Pozo}, in which higher smoothness of the two functions were showed, too.
	\begin{lemma}\label{lemma 3.5}
		Let $m \in (\mathcal {H})$ be satisfied and $\nu \geq 0$. Then
		\begin{itemize}
			\item[(i)] for arbitrary $t \in \mathbb R_+$ fixed, $\nu \mapsto s_\gamma(t,\nu)$ is differentiable. Moreover,
			\begin{align*}
				\partial_\nu s_\gamma(t,\nu) = -[s_\gamma(\cdot, \nu) \ast r_\gamma(\cdot, \nu)](t).
			\end{align*}
			
			\item[(ii)] for arbitrary $t \in (0,\infty)$ fixed, $\nu \mapsto r_\gamma(t,\nu)$ is differentiable. Moreover,
			\begin{align*}
				\partial_\nu r_\gamma(t,\nu) = -[r_\gamma(\cdot, \nu) \ast r_\gamma(\cdot, \nu)](t).
			\end{align*}
		\end{itemize}
		
		\begin{proof}
			(i) For arbitrary $t \geq 0$ and $\nu_0 \geq 0$ fixed, let $\delta>0$. Then \eqref{2.1} yields
			\begin{align}\label{3.2}
				\begin{aligned}
					&\quad s_\gamma(t,\nu_0+\delta) - s_\gamma(t,\nu_0) + \nu_0 [m_{\gamma}\ast(s_\gamma(\cdot, \nu_0+\delta) - s_\gamma(\cdot, \nu_0))](t) \\
					&= -\delta [m_{\gamma}\ast s_\gamma(\cdot,\nu_0+\delta)](t).
				\end{aligned}
			\end{align}
			This combined with Proposition \ref{prop 2.1} (i) shows that
			\begin{align*}
				\begin{aligned}
					s_\gamma(t,\nu_0+\delta) - s_\gamma(t,\nu_0)= \delta \Big(\nu_0[m_{\gamma}\ast s_\gamma(\cdot,\nu_0+\delta)\ast r_\gamma(\cdot,\nu_0)](t)- [m_{\gamma}\ast s_\gamma(\cdot,\nu_0+\delta)](t)\Big).
				\end{aligned}
			\end{align*}
			Since $|s_\gamma(t,\nu_0+\delta)|\leq1$ (due to Proposition \ref{prop 2.1} (iii)), we have
			\begin{align}\label{3.3}
				|s_\gamma(t,\nu_0+\delta) - s_\gamma(t,\nu_0)| \leq \delta \left[\nu_0(1 \ast m_{\gamma}\ast m_{\gamma})(t)+ (1 \ast m_{\gamma})(t)\right],
			\end{align}
			which implies that $\nu \mapsto s_\gamma(t,\nu)$ is right-continuous on $\mathbb R_+$. The left-continuity follows similarly, therefore $\nu \mapsto s_\gamma(t,\nu)$ is continuous on $\mathbb R_+$.
			
			Now we are ready to prove the differentiability of $\nu \mapsto s_\gamma(t,\nu)$. Let $s^{(1)}_\gamma(t,\nu_0)$ be the solution to the following Volterra integral equation
			\begin{align}\label{3.4}
				\begin{aligned}
					s^{(1)}_\gamma(t,\nu_0) + \nu_0 m_{\gamma}\ast s^{(1)}_\gamma(\cdot, \nu_0)(t)= - m_{\gamma}\ast s_\gamma(\cdot,\nu_0)(t).
				\end{aligned}
			\end{align}
			We denote
			$$b_\gamma(t,\nu_0,\delta) = \frac{s_\gamma(t,\nu_0+\delta) - s_\gamma(t,\nu_0)}{\delta}-s^{(1)}_\gamma(t,\nu_0).$$
			Using the equality \eqref{3.2}, we conclude that $b_\gamma(t,\nu_0)$ solves the equation
			\begin{align*}
				b_\gamma(t,\nu_0,\delta) + \nu_0 m_{\gamma}\ast b_\gamma(\cdot,\nu_0,\delta) = -m_{\gamma} \ast \left(s_\gamma(\cdot, \nu_0 + \delta) -s_\gamma(\cdot, \nu_0) \right)(t),
			\end{align*}
			it follows from Proposition \ref{prop 2.1} (i) again that
			\begin{align*}
				b_\gamma(t,\nu_0,\delta) &=\nu_0 [m_{\gamma} \ast \left(s_\gamma(\cdot, \nu_0 + \delta) -s_\gamma(\cdot, \nu_0) \right) \ast r_\gamma(\cdot,\nu_0)](t)\\
				&\quad -[m_{\gamma} \ast \left(s_\gamma(\cdot, \nu_0 + \delta) -s_\gamma(\cdot, \nu_0) \right)](t).
			\end{align*}
			Thanks to \eqref{3.3}, we obtain the inequality
			\begin{align*}
				|b_{\gamma}(t, \nu_0,\delta)| &\leq \delta \left[\nu_0^2(1 \ast m_{\gamma}\ast m_{\gamma}\ast m_{\gamma}\ast m_{\gamma})(t)+2\nu_0(1 \ast m_{\gamma}\ast m_{\gamma}\ast m_{\gamma})(t) \right.\\  &\left.\quad + (1 \ast m_{\gamma}\ast m_{\gamma})(t) \right].
			\end{align*}
			Therefore,
			\begin{align*}
				\lim_{\delta \to 0^+}\frac{s_\gamma(t,\nu_0+\delta) - s_\gamma(t,\nu_0)}{\delta}=s^{(1)}_\gamma(t,\nu_0).
			\end{align*}
			Similarly, $\partial_{\nu^{-} }s_\gamma(t,\nu)|_{\nu=\nu_0} = s^{(1)}_\gamma(t,\nu_0)$. Then $\nu \mapsto s_\gamma(t,\nu)$ is differentiable.
			Moreover, Proposition \ref{2.1}(i),  \eqref{2.2} and \eqref{3.4} allows to conclude that
			\begin{align*}
				s^{(1)}_\gamma(t,\nu_0) = -[s_\gamma(\cdot,\nu_0)\ast \left(m_{\gamma}- m_{\gamma}\ast r_\gamma(\cdot,\nu_0)\right)](t) = -[s_\gamma(\cdot,\nu_0)\ast r_\gamma(\cdot,\nu_0)](t).
			\end{align*}
			
			(ii) By the similar arguments, (ii) holds, too.
			The proof is completed.
		\end{proof}
	\end{lemma}

	\begin{lemma}\label{lemma 3.7}
		Let $m \in (\mathcal{H})$ be satisfied. Then for $t \in \mathbb R_+$ and $\varepsilon \in [0,1]$, we have
		\begin{align*}
				\left|\partial_\alpha s_\gamma(t,|\xi|^\alpha) \right| \leq  \left|\ln |\xi|\right| \quad {\it for}~ \xi \in \mathbb R^N \setminus\{0\}.
		\end{align*}
		Moreover, it holds
		\begin{align*}
			\left|\partial_\alpha s_\gamma(t,|\xi|^\alpha) \right| \geq \frac{1}{4} \frac{1}{1+[k_\gamma(t)]^{-1}}\frac{(1\ast m_\gamma)(t)}{1+(1\ast m_\gamma)(t)} \left| \ln |\xi| \right| \quad {\it for}~|\xi| \in [\frac{1}{2}, 1].
		\end{align*}
		\begin{proof}
			Let $|\xi| \neq 0$. We can easily get from Lemma \ref{lemma 3.5} (i) that
			\begin{align}\label{3.5}
				\partial_\alpha s_\gamma(t,|\xi|^\alpha)= - (s_\gamma(\cdot, |\xi|^\alpha) \ast r_\gamma(\cdot, |\xi|^\alpha))(t)|\xi|^\alpha \ln|\xi|.
			\end{align}
			Using Proposition \ref{prop 2.1} (ii) and (iii), we obtain
	\begin{align*}
					\left|\partial_\alpha s_\gamma(t,|\xi|^\alpha) \right|
					&\leq |\xi|^\alpha \left(1 \ast r_\gamma(\cdot, |\xi|^\alpha)\right)(t)\left|\ln |\xi|\right|
					=\left(1-s_\gamma(t,|\xi|^\alpha) \right)\left|\ln |\xi|\right| \leq \left|\ln |\xi|\right|.
			\end{align*}
			Then the first inequality follows.
			
			On the other hand, applying the nonincreasity of $k_\gamma$, \eqref{3.5}, Proposition \ref{prop 2.1} (ii) and (iii) again, we conclude that for $|\xi| \in [\frac{1}{2},1]$,
			\begin{align*}
				\left|\partial_\alpha s_\gamma(t,|\xi|^\alpha) \right|
				&\geq |\xi|^{\alpha} \left| \ln |\xi| \right| \int_0^t  \frac{1+\beta_\gamma |\xi|^\alpha[k_\gamma(\tau)]^{-1}r_\gamma(t,|\xi|^\alpha)}{1+|\xi|^\alpha [k_\gamma(\tau)]^{-1}} r_\gamma(t-\tau,|\xi|^\alpha) \, \mathrm{d}\tau\\
				&\geq |\xi|^{\alpha}\left| \ln |\xi| \right| \int_0^t  \frac{1}{1+|\xi|^\alpha [k_\gamma(\tau)]^{-1}} r_\gamma(t-\tau,|\xi|^\alpha) \, \mathrm{d}\tau\\
				&\geq \frac{1}{1+[k_\gamma(t)]^{-1}}\left| \ln |\xi| \right| |\xi|^{\alpha}\int_0^t r_\gamma(\tau,|\xi|^\alpha)\, \mathrm{d}\tau\\
				&\geq \frac{1}{1+[k_\gamma(t)]^{-1}}\frac{|\xi|^\alpha(1\ast m_\gamma)(t)}{1+|\xi|^\alpha(1\ast m_\gamma)(t)} \left| \ln |\xi| \right|\\
				&\geq \frac{1}{4} \frac{1}{1+[k_\gamma(t)]^{-1}}\frac{(1\ast m_\gamma)(t)}{1+(1\ast m_\gamma)(t)} \left| \ln |\xi| \right|.
			\end{align*}
			The proof is completed.
		\end{proof}
	\end{lemma}
	
	\begin{lemma}\label{lemma 3.8}
		Let $m$ satisfy $(\mathcal{H})$ and be nonincreasing. If $\sigma \in [0,\alpha)$ for $\alpha \in (1,2]$, then
		\begin{align*}
			|\xi|^{\sigma}\left|\partial_\alpha r_\gamma(t,|\xi|^\alpha) \right| \leq \frac{4\alpha}{\alpha-\sigma} m_\gamma \left(\frac{t}{2}\right) \left[(1\ast m_\gamma)\left(\frac{t}{2} \right)\right]^{-\frac{\sigma}{\alpha}} \left|\ln |\xi|\right|
		\end{align*}
		for $\xi \in \mathbb R^N \setminus\{0\}$ and $t \in (0,\infty)$.
		\begin{proof}
			For $|\xi| \neq 0$ and $\alpha \in (1,2]$, using Lemma \ref{lemma 3.5} (ii) we get
			\begin{align*}
				\partial_\alpha r_\gamma(t,|\xi|^\alpha)= - (r_\gamma(\cdot, |\xi|^\alpha) \ast r_\gamma(\cdot, |\xi|^\alpha))(t)|\xi|^\alpha \ln|\xi|.
			\end{align*}
			It follows from Proposition \ref{prop 2.1} (iii) that
			\begin{align*}
				|\xi|^{\sigma}|\partial_\alpha r_\gamma(t,|\xi|^\alpha)| &\leq  \left| \ln|\xi|\right| \int_0^t \frac{m_\gamma(t-\tau)|\xi|^{\frac{\alpha+\sigma}{2}}}{1+|\xi|^\alpha(1\ast m_\gamma)(t-\tau)} \frac{m_\gamma(\tau)|\xi|^{\frac{\alpha+\sigma}{2}}}{1+|\xi|^\alpha(1\ast m_\gamma)(\tau)}\, \mathrm{d}\tau\\
				&\leq \left| \ln|\xi|\right| \int_0^t m_\gamma(t-\tau)\left[(1\ast m_\gamma)(t-\tau)\right]^{-\frac{\alpha+\sigma}{2\alpha}} m_\gamma(\tau)\left[(1\ast m_\gamma)(\tau)\right]^{-\frac{\alpha+\sigma}{2\alpha}}\, \mathrm{d}\tau.
			\end{align*}
			It suffices to show
			\begin{align*}
				&\quad \int_0^t m_{\gamma}(t-\tau)[(1\ast m_{\gamma})(t-\tau)]^{-\frac{\alpha+\sigma}{2\alpha}}m_{\gamma}(\tau)[(1\ast m_{\gamma})(\tau)]^{-\frac{\alpha+\sigma}{2\alpha}} \, \mathrm{d}\tau\\
				&\leq m_\gamma\left(\frac{t}{2}\right)\left[(1\ast m_{\gamma})\left(\frac{t}{2}\right)\right]^{-\frac{\alpha+\sigma}{2\alpha}}\int_0^{\frac{t}{2}} [(1\ast m_{\gamma})(\tau)]^{-\frac{\alpha+\sigma}{2\alpha}}m_{\gamma}(\tau) \, \mathrm{d}\tau \\
				&\quad +m_\gamma \left(\frac{t}{2}\right)\left[(1\ast m_{\gamma})\left(\frac{t}{2}\right)\right]^{-\frac{\alpha+\sigma}{2\alpha}}\int_{\frac{t}{2}}^t [(1\ast m_{\gamma})(t-\tau)]^{-\frac{\alpha+\sigma}{2\alpha}}m_{\gamma}(t-\tau) \, \mathrm{d}\tau\\
				&\leq \frac{4\alpha}{\alpha-\sigma} m_\gamma \left(\frac{t}{2}\right) \left[(1\ast m_\gamma)\left(\frac{t}{2} \right)\right]^{-\frac{\sigma}{\alpha}}.
			\end{align*}
			The proof is completed.
		\end{proof}
	\end{lemma}

	\subsection{Estimates and continuity of fundamental solutions}
	%

In what follows, the proofs of  Sobolev estimates of $Q_{\gamma,\alpha}$ and $G_{\gamma,\alpha}$ are presented.
\begin{lemma}\label{lemma 3.1}
Let $m$ satisfy $(\mathcal{H})$ and $\alpha \in (1,2]$. Then for $s_1, s_2 \geq 0$ with $0\leq \frac{s_2-s_1}{\alpha} \leq 1$ and for $t>0$, we have
\begin{itemize}
	\item[{\rm(i)}] $\|Q_{\gamma,\alpha}(t, \cdot)\star u\|_{H^{s_2}} \leq 2^{\frac{s_2-s_1}{2}} \left(1+[t+\gamma(1\ast m)(t)]^{-\frac{s_2-s_1}{\alpha}}\right) \|u\|_{H^{s_1}}$.
	\item[{\rm(ii)}] $\|Q_{\gamma,\alpha}(t, \cdot)\star u\|_{\dot{H}^{s_2}} \leq [t+\gamma(1\ast m)(t)]^{-\frac{s_2-s_1}{\alpha}}\|u\|_{\dot{H}^{s_1}}$.
\end{itemize}
\begin{proof}
	(i)	Recall that $m_{\gamma}(t)=1+\gamma m(t)$. Since 	
	\begin{align*}
		\|Q_{\gamma,\alpha}(t, \cdot)\star u\|^2_{H^{s_2}}
		&=\int_{\mathbb R^N} s^2_\gamma(t,|\xi|^\alpha)|\tilde{u}(\xi)|^2(1+|\xi|^2)^{s_2}\, \mathrm{d}\xi\\
		&\leq \int_{\mathbb R^N} \frac{(1+|\xi|^2)^{s_2}}{\left(1+|\xi|^{\alpha}(1\ast m_{\gamma})(t)\right)^2}|\tilde{u}(\xi)|^2\, \mathrm{d}\xi\\
		&\leq \int_{\mathbb R^N} \frac{(1+|\xi|^2)^{s_2-s_1}}{\left(1+|\xi|^{\alpha}(1\ast m_{\gamma})(t)\right)^2}|\tilde{u}(\xi)|^2(1+|\xi|^2)^{s_1}\, \mathrm{d}\xi,
	\end{align*}
	where we have used Proposition \ref{prop 2.1} (iii) and \eqref{2.5}, using
	the fact that $0 \leq \frac{s_2-s_1}{\alpha} \leq 1$, we obtain
	\begin{align}\label{3.1}
		\begin{aligned}
			&\quad \frac{(1+|\xi|^2)^{\frac{s_2-s_1}{2}}}{1+|\xi|^{\alpha}(1\ast m_{\gamma})(t)}\\
			&\leq 2^{\frac{s_2-s_1}{2}} \frac{1+|\xi|^{s_2-s_1}}{1+|\xi|^{\alpha}(1\ast m_{\gamma})(t)}\\
			&\leq 2^{\frac{s_2-s_1}{2}}\left(1+ (1\ast m_{\gamma})(t)^{-\frac{s_2-s_1}{\alpha}} \frac{\left[|\xi|^{\alpha}(1\ast m_{\gamma})(t)\right]^{\frac{s_2-s_1}{\alpha}}}{1+|\xi|^{\alpha}(1\ast m_{\gamma})(t)} \right)\\
			&\leq 2^{\frac{s_2-s_1}{2}}\left(1+ (1\ast m_{\gamma})(t)^{-\frac{s_2-s_1}{\alpha}} \right).
		\end{aligned}
	\end{align}
	It shows that (i) holds.
	The inequality in (ii) follows from the similar arguments.
\end{proof}
\end{lemma}
\begin{lemma}\label{lemma 3.2}
Let $m$ satisfy $(\mathcal{H})$ and be nonincreasing. Then for $s_1, s_2 \geq 0$ with $0\leq \frac{s_2-s_1}{\alpha} \leq 1$ and for $t>0$, we have	
\begin{itemize}
	\item[{\rm(i)}] $\|G_{\gamma,\alpha}(t, \cdot)\star u\|_{H^{s_2}} \leq 2^{\frac{s_2-s_1}{2}}(1+\gamma m(t))\left(1+[t+\gamma(1\ast m)(t)]^{-\frac{s_2-s_1}{\alpha}}\right) \|u\|_{H^{s_1}}$.
	\item[{\rm(ii)}] $\|G_{\gamma,\alpha}(t, \cdot)\star u\|_{\dot{H}^{s_2}} \leq (1+\gamma m(t))[t+\gamma(1\ast m)(t)]^{-\frac{s_2-s_1}{\alpha}}\|u\|_{\dot{H}^{s_1}}$.
\end{itemize}
\begin{proof}
	We apply Proposition \ref{prop 2.1} (iii) and \eqref{2.7} to obtain that
	\begin{align*}
		\|G_{\gamma,\alpha}(t, \cdot)\star u\|^2_{H^{s_2}}
		&=\int_{\mathbb R^N} r^2_\gamma(t,|\xi|^\alpha)|\tilde{u}(\xi)|^2(1+|\xi|^2)^{s_2}\, \mathrm{d}\xi\\
		&\leq \int_{\mathbb R^N} \frac{m_{\gamma}(t)^2(1+|\xi|^2)^{s_2}}{\left(1+|\xi|^{\alpha}(1\ast m_{\gamma})(t)\right)^2}|\tilde{u}(\xi)|^2\, \mathrm{d}\xi\\
		&\leq \int_{\mathbb R^N} \frac{m_{\gamma}(t)^2(1+|\xi|^2)^{s_2-s_1}}{\left(1+|\xi|^{\alpha}(1\ast m_{\gamma})(t)\right)^2}|\tilde{u}(\xi)|^2(1+|\xi|^2)^{s_1}\, \mathrm{d}\xi.
	\end{align*}
	This combined with $0 \leq \frac{s_2-s_1}{\alpha} \leq 1$ and \eqref{3.1} shows the inequality in (i) holds.
	The inequality in (ii) follows from the similar arguments.
\end{proof}
\end{lemma}
\begin{remark}\label{lemma 3.4}
	For every $s\geq 0$ and $u\in H^s$, we can write $\|u\|_{H^s} = \|(I-\Delta)^{\frac{s}{2}}u\|_{L^2}$ and$\|u\|_{{\dot{H}}^s}=\|(-\Delta)^{\frac{s}{2}}u\|_{L^2}$. Assuming the condition in Lemma \ref{lemma 3.1} holds, we have the following inference:
	\begin{itemize}
		\item[{\rm(i)}] $\left\|(-\Delta)^{\frac{\alpha}{2}}Q_{\gamma,\alpha}(t,\cdot)\star u \right\|_{H^{s_1}} \leq [(1\ast m_\gamma)(t)]^{\frac{s_2-s_1}{\alpha}-1} \|u\|_{H^{s_2}}$ for $t>0$.
		\item[{\rm(ii)}] If $m $ is still nonincreasing, then
		\begin{align*}
			\left\|(-\Delta)^{\frac{\alpha}{2}}G_{\gamma,\alpha}(t,\cdot)\star u \right\|_{H^{s_1}} \leq m_\gamma(t)[(1\ast m_\gamma)(t)]^{\frac{s_2-s_1}{\alpha}-1} \|u\|_{H^{s_2}}, \quad t>0.
		\end{align*}
	\end{itemize}
	Indeed, it follows from Lemma \ref{lemma 3.1}(ii) that
	\begin{align*}
		&\quad \left\|(-\Delta)^{\frac{\alpha}{2}}Q_{\gamma,\alpha}(t,\cdot)\star u \right\|_{H^{s_1}} =\left\|(-\Delta)^{\frac{\alpha}{2}}(I-\Delta)^{\frac{s_1}{2}}Q_{\gamma,\alpha}(t,\cdot)\star u \right\|_{L^2}\\
		&=\left\|(I-\Delta)^{\frac{s_1}{2}}Q_{\gamma,\alpha}(t,\cdot)\star u \right\|_{{\dot H}^\alpha}=\left\|Q_{\gamma,\alpha}(t,\cdot)\star (I-\Delta)^{\frac{s_1}{2}}u \right\|_{{\dot H}^\alpha}\\
		&\leq [(1\ast m_\gamma)(t)]^{\frac{s_2-s_1}{\alpha}-1} \|(I-\Delta)^{\frac{s_1}{2}}u\|_{{\dot H}^{s_2-s_1}} \leq [(1\ast m_\gamma)(t)]^{\frac{s_2-s_1}{\alpha}-1} \|(I-\Delta)^{\frac{s_1}{2}}u\|_{ H^{s_2-s_1}}\\
		&=[(1\ast m_\gamma)(t)]^{\frac{s_2-s_1}{\alpha}-1} \|u\|_{H^{s_2}}.
	\end{align*}
	Similarly, Lemma \ref{lemma 3.2} (ii) follows (ii).
	\end{remark}
	
	Now we give the continuity of fundamental solutions $Q_{\gamma,\alpha}$ and $G_{\gamma,\alpha}$ in Sobolev spaces.

	\begin{lemma}\label{lemma 4.1}
Let the kernel $m$ satisfy $(\mathcal{H})$ and $s \geq 0$. Then the following assertions holds.
\begin{itemize}
	\item[{\rm(i)}] The operators $Q_{\gamma,\alpha}(t,\cdot)\star u$ and $(-\Delta)^{\frac{\alpha}{2}}Q_{\gamma,\alpha}(t,\cdot)\star u$ are continuous with respect to $t\in(0, \infty)$ in $H^{s}$. Moreover, we have
	\begin{align*}
		&\|Q_{\gamma,\alpha}(t_2, \cdot)\star u - Q_{\gamma, \alpha}(t_1,\cdot)\star u\|_{{H}^{s}}
		\leq \ln \frac{t_2+\gamma(1\ast m)(t_2)}{t_1+ \gamma(1\ast m)(t_1)}\|u\|_{{H}^{s}} \text{ ~for~ }  t_2 > t_1>0,
	\end{align*}
	and $Q(t,\cdot)\star u \to u$ as $t \to 0^+$.
	\item[{\rm(ii)}]If $m$ is still nonincreasing on $(0,\infty)$, then for  $s_2 >s_1\geq 0$ and $t_2>t_1>0$. It holds that
		\begin{align*}
			\left\|Q_{\gamma,\alpha}(t_2,\cdot)\star u- Q_{\gamma,\alpha}(t_1,\cdot)\star u\right\|_{H^{s_1}} \leq \frac{\alpha}{s_2-s_1}\left(\int_{t_1}^{t_2}m_\gamma(\tau)\, \mathrm{d}\tau \right)^{\frac{s_2-s_1}{\alpha}}\|u\|_{H^{s_2}}.
		\end{align*}
		Futhermore, we have
		\begin{align*}
			\|Q_{\gamma,\alpha}(t,\cdot)\star u-u\|_{H^{s_1}} \leq [1\ast m_\gamma(t)]^{\frac{s_2-s_1}{\alpha}}\|u\|_{H^{s_2}} \text{ ~for~ } t>0.
	\end{align*}
	\item[{\rm(iii)}] Assume further that $m$ is nonincreasing on $(0,\infty)$ for $\beta_\gamma>0$ or $m$ is nonincreasing and continuous for $\beta_\gamma=0$. Then the operators
	$G_{\gamma,\alpha}(t,\cdot)\star u$ and $(-\Delta)^{\frac{\alpha}{2}}G_{\gamma,\alpha}(t,\cdot)\star u$
	are continuous with respect to $t\in(0, \infty)$ in $H^{s}$.
\end{itemize}

\begin{proof}
	(i) Recall that $\tilde{Q}_{\gamma, \alpha}(t, \xi) = s_\gamma(t,|\xi|^\alpha)$. We apply Proposition \ref{prop 2.1} (ii) and (iii) to obtain that
	\begin{align*}
		&\quad\|Q_{\gamma,\alpha}(t_2, \cdot)\star u - Q_{\gamma,\alpha} (t_1, \cdot)\star u\|^2_{H^{s}}\\
		&=\int_{\mathbb R^N} \left(\int_{t_1}^{t_2}|\xi|^{\alpha}r_\gamma(\tau,|\xi|^\alpha)\, \mathrm{d}\tau \right)^2 |\tilde{u}(\xi)|^2(1+|\xi|^2)^{s}\, \mathrm{d}\xi\\
		&\leq \int_{\mathbb R^N} \left(\int_{t_1}^{t_2} \frac{m_{\gamma}(\tau)|\xi|^\alpha}{1+|\xi|^{\alpha}(1\ast m_{\gamma})(\tau)}\, \mathrm{d}\tau \right)^2 |\tilde{u}(\xi)|^2(1+|\xi|^2)^{s}\, \mathrm{d}\xi\\
		&\leq \left(\int_{t_1}^{t_2} m_{\gamma}(\tau)\left[(1\ast m_{\gamma})(\tau)\right]^{-1}\, \mathrm{d}\tau \right)^2 \|u\|^2_{H^{s}}\\
		&=
		\left[\ln \frac{(1\ast m_{\gamma})(t_2)}{ (1\ast m_{\gamma})(t_1)}\right]^2\|u\|^2_{H^{s}},
	\end{align*}
	this yields the continuity of $Q_{\gamma,\alpha}(t,\cdot)\star u$ in $H^s$ and also ensures the inequality in (i).
	
	To prove that $
	\lim_{t\to 0^+}\left\|Q_{\gamma,\alpha}(t, \cdot)\star u - u\right\|^2_{H^s}=0$,
	we will use the Lebesgue's dominated convergence theorem. Therefore,
	we needs to verify that
	\begin{align*}
		(s_\gamma(t, |\xi|^\alpha)-1)^2 |\tilde{u}(\xi)|^2(1+|\xi|^2)^{s}\leq 4 |\tilde{u}(\xi)|^2(1+|\xi|^2)^{s},\quad \xi \in \mathbb R^N
	\end{align*}
	and $\lim_{t \to 0^+}s_\gamma(t, |\xi|^\alpha)=1$. Clearly, it easily follows from  Proposition \ref{prop 2.1} (ii).

	It remains to show the continuity of $(-\Delta)^{\frac{\alpha}{2}}Q_{\gamma,\alpha}(t,\cdot)\star u$ in $H^s$.  Fix $t \in (0, \infty)$. Let $0<|h|<\frac{t}{2}$. It follows from Proposition \ref{prop 2.1} (iii) that
	\begin{align*}
		|\xi|^\alpha s_\gamma(t,|\xi|^\alpha) \leq
		[(1 \ast m_{\gamma})(t)]^{-1}.
	\end{align*}
	Using this and nonnegativity of $m$, we arrive at
	\begin{align*}
		&\quad|\xi|^{2\alpha}\left[s_\gamma(t+h,|\xi|^\alpha) - s_\gamma(t,|\xi|^\alpha) \right]^2 |\tilde{u}(\xi)|^2(1+|\xi|^2)^s \\
		&\leq 4\left[(1 \ast m_{\gamma})\left(\frac{t}{2}\right)\right]^{-2}|\tilde{u}(\xi)|^2(1+|\xi|^2)^s.
	\end{align*}
	This implies that	
	\begin{align*}
		&\quad \lim_{h\to 0} \left\|(-\Delta)^{\frac{\alpha}{2}}Q_{\gamma,\alpha}(t+h, \cdot)\star u-(-\Delta)^{\frac{\alpha}{2}}Q_{\gamma,\alpha}\star u\right\|^2_{H^s}\\
		&= \int_{\mathbb R^N} \lim_{h\to 0}|\xi|^{2\alpha}\left[s_\gamma(t+h,|\xi|^\alpha) - s_\gamma(t,|\xi|^\alpha) \right]^2 |\tilde{u}(\xi)|^2(1+|\xi|^2)^s\, \mathrm{d}\xi=0,
	\end{align*}
	where we have used Proposition \ref{prop 2.1} (ii) and the Lebesgue's dominated convergence theorem again.
	
	(ii) From Proposition \ref{prop 2.1} (ii) and (iii) we obtain,
		\begin{align*}
			&\quad\|Q_{\gamma,\alpha}(t_2, \cdot)\star u - Q_{\gamma,\alpha} (t_1, \cdot)\star u\|^2_{H^{s_1}}\\
			&=\int_{\mathbb R^N} \left(\int_{t_1}^{t_2}|\xi|^{\alpha}r_\gamma(\tau,|\xi|^\alpha)\, \mathrm{d}\tau \right)^2 |\tilde{u}(\xi)|^2(1+|\xi|^2)^{s_1}\, \mathrm{d}\xi\\
			&\leq \int_{\mathbb R^N} \left(\int_{t_1}^{t_2} \frac{m_{\gamma}(\tau)|\xi|^{\alpha-(s_2-s_1)}}{1+|\xi|^{\alpha}(1\ast m_{\gamma})(\tau)}\, \mathrm{d}\tau \right)^2 |\tilde{u}(\xi)|^2(1+|\xi|^2)^{s_2}\, \mathrm{d}\xi\\
			&\leq \left(\int_{t_1}^{t_2} m_\gamma(\tau) \left[(1\ast m_\gamma)(\tau)\right]^{\frac{s_2-s_1}{\alpha}-1}\, \mathrm{d}\tau \right)^2\|u\|^2_{H^{s_2}}\\
			&\leq \frac{\alpha}{s_2-s_1}\left(\int_{t_1}^{t_2}m_\gamma(\tau)\, \mathrm{d}\tau \right)^{2\frac{s_2-s_1}{\alpha}}\|u\|^2_{H^{s_2}}.
		\end{align*}
		Similar proof can lead to the second assertion of (ii).
	
	(iii) Recalling Proposition \ref{prop 2.1} (iii), we observe that
	$r_\gamma(t,|\xi|^\alpha) \leq m_{\gamma}(t)$ and
	\begin{align}\label{4.2}
		|\xi|^\alpha r_\gamma(t,|\xi|^\alpha) \leq
		m_{\gamma}(t)[(1 \ast m_{\gamma})(t)]^{-1},
	\end{align}
	where we have also used the nonincreasity of $m$ in \eqref{4.2}. Then it easily check that
	\begin{align*}
		\left[r_\gamma(t+h,|\xi|^\alpha) - r_\gamma(t,|\xi|^\alpha) \right]^2 |\tilde{u}(\xi)|^2(1+|\xi|^2)^s
		\leq 4\left[m_{\gamma}\left(\frac{t}{2}\right)\right]^{2}|\tilde{u}(\xi)|^2(1+|\xi|^2)^s
	\end{align*}
	and
	\begin{align*}
		&\quad|\xi|^{2\alpha}\left[r_\gamma(t+h,|\xi|^\alpha) - r_\gamma(t,|\xi|^\alpha) \right]^2 |\tilde{u}(\xi)|^2(1+|\xi|^2)^s \\
		&\leq 4\left[m_{\gamma}\left(\frac{t}{2}\right)\right]^{2}\left[(1 \ast m_{\gamma})\left(\frac{t}{2}\right)\right]^{-2}|\tilde{u}(\xi)|^2(1+|\xi|^2)^s.
	\end{align*}
	Noting the continuity of $r_\gamma(t,|\xi|^\alpha)$ with respect to $t$ (thanks to Lemma \ref{newprop 2.2}), we can conclude the continuity of $G_{\gamma,\alpha}(t,\cdot)\star u$ and $(-\Delta)^{\frac{\alpha}{2}}G_{\gamma,\alpha}(t,\cdot)\star u$ as the same way as we derived (i).
\end{proof}
\end{lemma}

\section{Classical solution and convergence of linear problem}
In this section, we consider the classical solutions of the linear problem
\begin{align}\label{neweq 1.2}
\begin{cases}
	\partial_{t}v_\alpha + (-\Delta)^{\frac{\alpha}{2}}
	v_\alpha + \gamma \partial_{t}(m \ast (-\Delta)^{\frac{\alpha}{2}} v_\alpha) =g(t,x), \quad t>0, \ x \in \mathbb R^N, \\
	v_{\alpha}|_{t=0}=v_{0,\alpha}(x), \quad  x \in \mathbb R^N.
\end{cases}
\end{align}
where $g:(0,T]\times \mathbb R^N \mapsto \mathbb R$ is a function.
In what follows, we first the concept of classical solutions to the problem
\eqref{neweq 1.2}.
\begin{definition}\label{def 4.1}
For $s\geq 0$ and $v_{0,\alpha} \in H^s$. Let $g(t):(0,T]\mapsto H^s$ for $T>0$. The function $v_\alpha(t,x)$ is called a classical solution to the problem \eqref{neweq 1.2} in $H^s$, if and only if $v_\alpha \in C\left([0,T]; H^s \right)$ with $(-\Delta)^{\frac{\alpha}{2}}v_\alpha \in H^s$, and  $\partial_t v_\alpha,~\partial_t (m\ast (-\Delta )^{\frac{\alpha}{2}}v_\alpha) \in  C\left((0,T]; H^s \right)$. Moreover, $v_\alpha$ satisfies the integral equation
\begin{align}\label{4.1}
	v_\alpha(t)= Q_{\gamma,\alpha}(t,\cdot) \star v_{0, \alpha} + \int_0^t Q_{\gamma,\alpha}(t-\tau, \cdot) \star g(s) \mathrm{d}s, \quad (t,x) \in \mathbb (0, T] \times \mathbb R^N.
\end{align}
\end{definition}

Put $D((-\Delta)^{\frac{\alpha}{2}}):=\{u: u \in H^s, (-\Delta)^{\frac{\alpha}{2}}u \in H^{s}\}$. In fact, we also write $D((-\Delta)^{\frac{\alpha}{2}})=H^{s+\alpha}$. Clearly, $(-\Delta)^{\frac{\alpha}{2}}: H^{s+\alpha} \to H^{s}$ is a bounded linear operator.

\subsection{Existence of Classical solutions}	
In the following lemma, we collect the regularity results of the fundamental solutions.
\begin{lemma}\label{lemma 4.2}
Let $m\in (\mathcal{H})$ and $u \in H^s$ for $s \geq 0$. Then  $(-\Delta)^{\frac{\alpha}{2}} Q_{\gamma,\alpha}(t, \cdot) \star u \in H^s$ for $t \in (0,\infty)$. Assume further that $m$ is nonincreasing on $(0,\infty)$. Then the following assertions holds.
\begin{itemize}
	\item[(i)] We have $(-\Delta)^{\frac{\alpha}{2}} G_{\gamma,\alpha}(t, \cdot) \star u \in H^s$ and $\partial_t \left[Q_{\gamma,\alpha}(t, \cdot)\star u\right]=-(-\Delta)^{\frac{\alpha}{2}}G_{\gamma,\alpha}(t,\cdot)\star u$ for $t \in (0,\infty)$.
	\item[(ii)] When $\beta_\gamma>0$ or $\beta_\gamma=0$ and $m$ is continuous, we have $\partial_t \left( m_{\gamma} \ast Q_{\gamma,\alpha}(\cdot, \cdot)\star u\right)(t)=G_{\gamma,\alpha}(t,\cdot)\star u$ for $t \in (0,\infty)$.
\end{itemize}
\begin{proof}
	Let $u \in H^s$ and $t \in (0,\infty)$. Using Proposition \ref{prop 2.1} (iii) we get
	\begin{align*}
		&\quad \left\|(-\Delta)^{\frac{\alpha}{2}}Q_{\gamma,\alpha} \star u \right\|_{H^s} \\
		&= \left(\int_{\mathbb R^N}|\xi|^{2\alpha}s^2_\gamma(t, |\xi|^\alpha) |\tilde{u}(\xi)|^2(1+|\xi|^2)^{s}\, \mathrm{d}\xi \right)^{\frac{1}{2}}\\
		&\leq [(1\ast m_{\gamma})(t)]^{-1} \left(\int_{\mathbb R^N}\frac{|\xi|^{2\alpha}[(1\ast m_{\gamma})(t)]^2}{\left(1+|\xi|^\alpha(1\ast m_{\gamma})(t)\right)^2}|\tilde{u}(\xi)|^2(1+|\xi|^2)^{s} \, \mathrm{d}\xi \right)^{\frac{1}{2}}\\
		&\leq [(1\ast m_{\gamma})(t)]^{-1}\|u\|_{H^s},
	\end{align*}
	which implies $(-\Delta)^{\frac{\alpha}{2}} Q_{\gamma,\alpha}(t, \cdot) \star u \in H^s$ for $t \in (0,\infty)$.
	
	(i)  Since $m$ is nonincreasing on $(0,\infty)$, it is easy to verify that $(-\Delta)^{\frac{\alpha}{2}} G_{\gamma,\alpha}(t, \cdot) \star u \in H^s$ for $t \in (0,\infty)$ by the similar arguments.

	Fix $t \in (0, \infty)$, we let $0<|h|<\frac{t}{2}$. Since
	\begin{align*}
		\frac{1}{h}\left[Q_{\gamma,\alpha}(t+h, \cdot)-Q_{\gamma,\alpha}(t, \cdot)\right]\star u=\mathcal F^{-1}\left(\frac{1}{h} \left[s_\gamma(t+h, |\xi|^\alpha)-s_\gamma(t, |\xi|^\alpha)\right] \tilde {u}(\xi)\right).
	\end{align*}
	Applying Proposition \ref{prop 2.1} (ii), one has
	\begin{align*}
		\lim\limits_{h \to 0} \frac{1}{h} \left[s_\gamma(t+h, |\xi|^\alpha)-s_\gamma(t, |\xi|^\alpha)\right] \tilde {u}(\xi)
		=-|\xi|^\alpha r_\gamma(t,|\xi|^\alpha)\tilde{u}(\xi),
	\end{align*}
	and there exists a $t_1$ in the range of $t$ to $t+h$ such that
	\begin{align*}
		\frac{1}{h} |\left[s_\gamma(t+h, |\xi|^\alpha)-s_\gamma(t, |\xi|^\alpha)\right] \tilde {u}(\xi)|
		&=\left||\xi|^\alpha r_\gamma(t_1,|\xi|^\alpha)\tilde{u}(\xi) \right|\\
		&\leq m_{\gamma}\left(\frac{t}{2}\right)\left[(1 \ast m_{\gamma})\left(\frac{t}{2}\right)\right]^{-1}
		\left|\tilde{u}(\xi)\right|,
	\end{align*}
	where the above inequality follows from the mean value theorem of the integral and \eqref{4.2}. Therefore, the Lebesgue's dominated convergence theorem implies that
	\begin{align*}
		\lim\limits_{h \to 0}\frac{1}{h}\left(Q_{\gamma,\alpha}(t+h, \cdot)-Q_{\gamma,\alpha}(t, \cdot)\right)\star u
		=-(-\Delta)^{\frac{\alpha}{2}}G_{\gamma,\alpha}(t,\cdot)\star u~~{\rm in}~H^s.
	\end{align*}
	Then the assertion (i) holds.
	
	(ii)  Noting the strong continuity of the function $G_{\gamma,\alpha}$ with respect to $t$ on $(0,\infty)$ (thanks to Lemma \ref{lemma 4.1} (iii)), we use immediately Proposition \ref{prop 2.2} (ii) to know that
	\begin{align*}
		\partial_t \left( m_{\gamma} \ast Q_{\gamma,\alpha}(\cdot, \cdot)\star u\right)(t) =\partial_t (1 \ast G_{\gamma,\alpha}(\cdot, \cdot)\star u)(t) = G_{\gamma,\alpha}(t, \cdot)\star u ~~{\rm in}~H^s.
	\end{align*}
	We complete the proof.
\end{proof}
\end{lemma}
\begin{remark}\label{nwlemma 3.4}
Let the conditions of Lemma \ref{lemma 4.2} be satisfied. If $u \in D((-\Delta)^{\frac{\alpha}{2}})$, then
$(-\Delta)^{\frac{\alpha}{2}} Q_{\gamma,\alpha}(t, \cdot) \star u=Q_{\gamma,\alpha}(t, \cdot) \star (-\Delta)^{\frac{\alpha}{2}} u$ and $(-\Delta)^{\frac{\alpha}{2}} Q_{\gamma,\alpha}(t, \cdot) \star u=Q_{\gamma,\alpha}(t, \cdot) \star (-\Delta)^{\frac{\alpha}{2}} u$ for $t\in(0,\infty)$. Moreover, we have
$\partial_t \left[Q_{\gamma,\alpha}(t, \cdot)\star u\right]=-G_{\gamma,\alpha}(t,\cdot)\star (-\Delta)^{\frac{\alpha}{2}}u$ for $t\in(0,\infty)$.
\end{remark}

Now we are ready to present the main result.
\begin{theorem}\label{the 4.1}
Let the kernel $m\in(\mathcal{H})$  be nonincreasing
and continuous on $(0,\infty)$. For $s\geq 0$, $T>0$ and $\theta_1>0$, we assume that $g:(0,T] \to H^{s+\theta_1}$ is bounded and continuous. Then, for every $v_{0,\alpha} \in H^{s+\theta_2}$ with $\theta_2>0$, the problem \eqref{neweq 1.2} admits a unique classical solution.
\end{theorem}
\begin{proof}
\emph{Step 1.} For $T>0$, we set $v_{\alpha,1}(t):=Q_{\gamma,\alpha}(t, \cdot) \star v_{0,\alpha}(x)$ for $t\in (0,T],~x\in \mathbb R^N$. We prove that $v_{\alpha,1}(t)$ is a classical solution to the following equation
\begin{align}\label{4.3}
	\begin{cases}
		\partial_t v(t)+ (-\Delta)^{\frac{\alpha}{2}} v(t)+ \gamma\partial_t m \ast (-\Delta)^{\frac{\alpha}{2}} v(t)=0, \quad t \in (0,T],\\
		v|_{t=0}=v_{0,\alpha}.
	\end{cases}
\end{align}
Notice $H^{s+\theta} \subset H^s$ for any $\theta>0$. It yields immediately from Lemma \ref{lemma 4.1} (i) that $v_{\alpha,1} \in C\left([0,T]; H^s \right)$ and $(-\Delta)^{\frac{\alpha}{2}}v_{\alpha,1} \in  C\left((0,T]; H^s \right)$. Then $v_{\alpha,1}(t)\in D((-\Delta)^{\frac{\alpha}{2}})$ for $t\in(0,T]$. Applying Lemma \ref{lemma 4.1} (iii) and Lemma \ref{lemma 4.2} (i), we obtain
\begin{align}\label{4.4}
	\partial_t v_{\alpha,1}(t) = - (-\Delta)^{\frac{\alpha}{2}}G_{\gamma,\alpha}(t,\cdot)\star v_{0,\alpha}
\end{align}
and $\partial_t v_{\alpha,1}\in  C\left((0,T]; H^s\right)$.

Next, we claim that $\partial_t (m_{\gamma}\ast (-\Delta)^{\frac{\alpha}{2}} v_{\alpha,1}) \in  C\left((0,T]; H^s\right)$. In fact, it suffices to verify the following relation
\begin{align}\label{4.5}
	[m_{\gamma}\ast (-\Delta)^{\frac{\alpha}{2}} v_{\alpha,1}](t)
	=[1\ast (-\Delta)^{\frac{\alpha}{2}}G_{\gamma,\alpha}(\cdot,\cdot) \star v_{0,\alpha}](t) \quad{\rm for}~ t>0.
\end{align}
We consider the right term of \eqref{4.5}. For $v_{0,\alpha} \in H^{s+\theta_2}$, let us do this in two cases: $\theta_2\in(0,\alpha]$ and $\theta_2\in(\alpha,\infty)$.
For $\theta_2\in(0,\alpha]$, we use Lemma \ref{lemma 3.2} (i) to calculate
\begin{align*}
	\int_0^t \left\|G_{\gamma,\alpha}(\tau,\cdot)\star v_{0,\alpha}\right\|_{H^{s+\alpha}}\, \mathrm{d}\tau &\leq 2^{\frac{\alpha-\theta}{2}}\|v_{0,\alpha}\|_{H^{s+\theta_2}}\int_0^t m_{\gamma}(\tau)\left(1+[(1\ast m_{\gamma})(\tau)]^{\frac{\theta_2}{\alpha}-1}\right) \, \mathrm{d}\tau\\
	&\leq 2^{\frac{\alpha-\theta_2}{2}} \|v_{0,\alpha}\|_{H^{s+\theta_2}} \left((1\ast m_\gamma)(t) +\frac{\alpha}{\theta_2}[(1\ast m_{\gamma})(t)]^{\frac{\theta_2}{\alpha}}\right).
\end{align*}
For $\theta_2\in(\alpha,\infty)$, we observe that $H^{s+\theta_2}\hookrightarrow H^{s+\alpha}$.
Therefore, $G_{\gamma,\alpha} \star v_{0,\alpha} \in L^1([0,t];H^{s+\alpha})$ for $\theta_2>0$. Since the operator $(-\Delta)^{\frac{\alpha}{2}}: H^{s+\alpha} \to H^{s}$ is linear and bounded, we have
\begin{align}\label{4.6}
	[1\ast (-\Delta)^{\frac{\alpha}{2}}G_{\gamma,\alpha}(\cdot,\cdot) \star v_{0,\alpha}](t)=(-\Delta)^{\frac{\alpha}{2}}[1\ast G_{\gamma,\alpha}(\cdot,\cdot) \star v_{0,\alpha}](t)  \quad{\rm for}~t>0.
\end{align}
We continue \eqref{4.6} with Lemma \ref{lemma 4.2} (i) to show
\begin{align}\label{4.7}
	[1\ast (-\Delta)^{\frac{\alpha}{2}}G_{\gamma,\alpha}(\cdot,\cdot) \star v_{0,\alpha}](t)
	=(-\Delta)^{\frac{\alpha}{2}}[m_{\gamma}\ast  Q_{\gamma,\alpha}(\cdot,\cdot) \star v_{0,\alpha}](t).
\end{align}

On the other hand, we consider the left term of \eqref{4.5}. Similarly, we only need to consider $v_{0,\alpha} \in H^{s+\theta_2}$ for $\theta_2\in(0,\alpha]$. Using the nonincreasity of the kernel $m$ and Lemma \ref{lemma 3.1} (i), we calculate that
\begin{align*}
	&\quad \int_0^t m_{\gamma}(t-\tau)\|Q_{\gamma,\alpha}(\tau,\cdot) \star v_{0,\alpha}\|_{H^{s+\alpha}}\, \mathrm{d}\tau\\
	&\leq 2^{\frac{\alpha-\theta_2}{2}}\| v_{0,\alpha}\|_{H^{s+\theta_2}}\int_0^t m_{\gamma}(t-\tau)\left(1+ [(1\ast m_{\gamma})(\tau)]^{\frac{\theta_2}{\alpha}-1}\right)\, \mathrm{d}\tau\\
	&\leq 2^{\frac{\alpha-\theta_2}{2}}\| v_{0,\alpha}\|_{H^{s+\theta_2}}\left(\int_0^{\frac{t}{2}} m_{\gamma}(\tau)\left(1+[(1\ast m_{\gamma})(\tau)]^{\frac{\theta_2}{\alpha}-1}\right)\, \mathrm{d}\tau \right.\\
	&\quad+ \left.\int_{\frac{t}{2}}^{t} m_{\gamma}(t-\tau)\left(1+[(1\ast m_{\gamma})(\tau)]^{\frac{\theta_2}{\alpha}-1}\right)\, \mathrm{d}\tau\right)\\
	&\leq 2^{\frac{\alpha-\theta_2}{2}}\| v_{0,\alpha}\|_{H^{s+\theta_2}}\left(2(1\ast m_\gamma)\left(\frac{t}{2}\right) + \left( 1+\frac{\alpha}{\theta_2} \right)\left[1\ast m_{\gamma}\left(\frac{t}{2}\right)\right]^{\frac{\theta_2}{\alpha}}\right).
\end{align*}
As we derived \eqref{4.6}, it also holds
\begin{align}\label{4.8}
	(-\Delta)^{\frac{\alpha}{2}}[m_{\gamma}\ast  Q_{\gamma,\alpha}(\cdot,\cdot) \star v_{0,\alpha}](t)=[m_{\gamma}\ast  (-\Delta)^{\frac{\alpha}{2}}Q_{\gamma,\alpha}(\cdot,\cdot) \star v_{0,\alpha}](t)
\end{align}
for $v_{0,\alpha} \in H^{s+\theta_2}$ with $\theta_2>0$.
Combing  \eqref{4.7} and \eqref{4.8}, the relation \eqref{4.5} holds. Then it implies
that		\begin{align}\label{4.9}
	\partial_t [m_{\gamma}\ast (-\Delta)^{\frac{\alpha}{2}} v_{\alpha,1}]
	=(-\Delta)^{\frac{\alpha}{2}}G_{\gamma,\alpha}(\cdot,\cdot) \star v_{0,\alpha}.
\end{align}	
Thanks to $\partial_t [m\ast (-\Delta)^{\frac{\alpha}{2}} v_{\alpha,1}] =\frac{1}{\gamma}\partial_t [m_{\gamma}\ast (-\Delta)^{\frac{\alpha}{2}} v_{\alpha,1}]- (-\Delta)^{\frac{\alpha}{2}} v_{\alpha,1}$, we have
$$
\partial_t [m\ast (-\Delta)^{\frac{\alpha}{2}} v_{\alpha,1}] \in  C\left((0,T]; H^s \right).
$$
This together with \eqref{4.4} and \eqref{4.9} shows that $v_{\alpha,1}$ is a classical solution to the equation \eqref{4.3}.

\emph{Step 2.} Let us put $v_{\alpha,2}(t)=\int_0^t Q_{\gamma,\alpha}(t-\tau, \cdot) \star g(\tau)\, \mathrm{d}\tau, \ t>0$. Then we need to prove that $v_{\alpha,2}$ is also a classical solution to the following equation
\begin{align}\label{4.10}
	\begin{cases}
		\partial_t v(t)+ (-\Delta)^{\frac{\alpha}{2}} v(t)+ \gamma \partial_t m \ast (-\Delta)^{\frac{\alpha}{2}}v(t)=g(t), \quad t \in (0,T],\\
		v|_{t=0}=0.
	\end{cases}
\end{align}
To do this, we firstly show $v_{\alpha,2} \in C\left([0,T]; H^s \right) \bigcap C^1\left((0,T]; H^s \right)$. In fact, since
\begin{align*}
	\left\|\int_0^t Q_{\gamma,\alpha}(t-\tau, \cdot)\star g(\tau)\, \mathrm{d}\tau\right\|_{H^s}
	&\leq \int_0^t \|Q_{\gamma,\alpha}(t-\tau, \cdot)\star g(\tau) \|_{H^s}\, \mathrm{d}\tau\\
	&\leq 2\int_0^t \|g(\tau)\|_{H^s}\, \mathrm{d}\tau \leq 2 t \sup_{\zeta \in (0,T]} \|g(\zeta)\|_{H^{s+\theta_1}} \to 0 \text{ ~as~ } t \to 0^+,
\end{align*}		
it follows from	the assumptions on $g$ and Lemma \ref{lemma 3.1} (i) that $v_{\alpha,2} \in C\left([0,T]; H^s \right)$.

It remains to verify that $\partial_t v_{\alpha,2} \in  C\left((0,T]; H^s \right)$. For arbitrary $t\in (0,T)$ fixed, we choose $0<h<\min\{\frac{t}{2}, \frac{T-t}{2}\}$. Let us consider
\begin{align*}
	\frac{v_{\alpha,2}(t+h)-v_{\alpha,2}(t)}{h}
	&=\int_0^{t}\frac{1}{h}\left[Q_{\gamma,\alpha}(t+h-\tau, \cdot)-Q_{\gamma,\alpha}(t-\tau, \cdot)\right]\star g(\tau)\, \mathrm{d}\tau\\ &\quad +\frac{1}{h}\int_{t}^{t+h} Q_{\gamma,\alpha}(t+h-\tau, \cdot)\star g(\tau)\, \mathrm{d}\tau\\
	&=: I_5(t,h)+I_6(t,h).
\end{align*}
As for $g(t)\in H^{s+\theta_1}$,
in view of the previous arguments on $v_{0,\alpha}$, we only debate the case $\theta_1\in(0, \alpha]$. For $I_5(t,h)$, because of the way we have dealt with Lemma \ref{lemma 4.2} (i), we
have
\begin{align*}
	\begin{aligned}
		&\quad \Big\|\frac{1}{h}[Q_{\gamma,\alpha}(t+h-\tau, \cdot)-Q_{\gamma,\alpha}(t-\tau, \cdot)]\star g(\tau)\Big\|_{H^s} \\
		&\leq m_{\gamma}(t-\tau)\left[(1 \ast m_{\gamma})(t-\tau)\right]^{\frac{\theta_1}{\alpha}-1} \sup_{\zeta \in (0,T]}\|g(\zeta)\|_{H^{s+\theta_1}}.
	\end{aligned}
\end{align*}
Combining this, Lemma \ref{lemma 4.2} (i) with the Lebesgue's dominated convergence theorem, we conclude that
\begin{align*}
	\lim_{h\to 0^+} I_5(t,h) = - \int_0^t (-\Delta)^{\frac{\alpha}{2}}G_{\gamma,\alpha}(t-\tau, \cdot) \star g(\tau)\, \mathrm{d}\tau
\end{align*}
for $\theta_1>0$. For convenience we introduce the notation  $v_{\alpha,3}(t)=:- \int_0^t (-\Delta)^{\frac{\alpha}{2}}G_{\gamma,\alpha}(t-\tau, \cdot) \star g(\tau)\, \mathrm{d}\tau$.
For $I_6(t,h)$, we divide it into two parts
\begin{align*}
	I_6 (t,h)&= \frac{1}{h}\int_{0}^{h} Q_{\gamma,\alpha}(h-\tau, \cdot)\star g(t+\tau)\, \mathrm{d}\tau\\
	&=\frac{1}{h}\int_{0}^{h} \left[Q_{\gamma,\alpha}(h-\tau, \cdot)-1\right]\star g(t+\tau)\, \mathrm{d}\tau+\frac{1}{h}\int_{0}^{h}g(t+\tau)\, \mathrm{d}\tau.
\end{align*}
Notice that $H^{s+\theta_1}\hookrightarrow H^{s}$ for $\theta_1>0$ and the fact that
\begin{align*}
	&\quad\left\|\frac{1}{h}\int_{0}^{h} \left[Q_{\gamma,\alpha}(h-\tau, \cdot) - 1\right]\star g(t+h)\, \mathrm{d}\tau \right\|_{H^s} \\
	&\leq \frac{1}{h}\int_0^h [(1 \ast  m_\gamma)(h-\tau)]^{\frac{\theta_1}{\alpha}}\|g(t+h)\|_{H^{s+\theta_1}}\, \mathrm{d}\tau\\
	&\leq [(1 \ast  m_\gamma)(h)]^{\frac{\theta_1}{\alpha}}\sup_{\zeta \in (0,T]}\|g(\zeta)\|_{H^{s+\theta_1}}
	\to 0 \text{ ~as~ } h \to 0^+,
\end{align*}
where we have used Lemma \ref{lemma 4.1} (ii). This together with the continuity of
$g$ yields that $\lim_{h\to 0^+} I_6(t,h)= g(t)$. Thus we arrive at
\begin{align*}
	\partial_{t^+} v_{\alpha,2}(t) = v_{\alpha,3}(t)+g(t)~~{\rm for}~t\in(0,T].
\end{align*}
So similarly, we find $\partial_{t^-} v_{\alpha,2}(t) =v_{\alpha,3}(t)+g(t)$. Hence,
\begin{align}\label{4.11}
	\partial_{t} v_{\alpha,2}(t) = v_{\alpha,3}(t)+g(t) \quad{\rm for}~ t \in (0,T].
\end{align}

We further prove $v_{\alpha,3} \in C((0,T]; H^s)$. For arbitrary $t\in (0,T)$ fixed, we choose $0<h<\min\{\frac{t}{2}, \frac{T-t}{2}\}$, then consider
\begin{align*}
	{v_{\alpha,3}(t+h)-v_{\alpha,3}(t)}
	&=-\int_0^{t} (-\Delta)^{\frac{\alpha}{2}} \left[G_{\gamma,\alpha}(t+h-\tau, \cdot)-G_{\gamma,\alpha}(t-\tau, \cdot)\right]\star g(\tau)\, \mathrm{d}\tau\\
	&\quad -\int_{t}^{t+h} (-\Delta)^{\frac{\alpha}{2}}G_{\gamma,\alpha}(t+h-\tau, \cdot)\star g(\tau)\, \mathrm{d}\tau\\
	&=: I_7(t,h)+I_8(t,h).
\end{align*}
We deal with $I_7(t,h)$ and $I_8(t,h)$, respectively. For $I_8(t,h)$, we observe that \begin{align*}
	\|I_8(t,h)\|_{H^s}
	&\leq \int_t^{t+h} m_\gamma(t+h-\tau) \|g(\tau)\|_{H^s} \, \mathrm{d}\tau\\
	&\leq \sup_{\zeta \in (0,T]} \|g(\zeta)\|_{H^s} (1\ast m_\gamma)(h) \to 0 \text{ ~as~ } h \to 0^+.
\end{align*}
For $I_7(t,h)$, we can continue it in two cases: $\theta_1\in(0,\alpha]$ and $\theta_1\in(\alpha,\infty)$. Let $\theta_1\in(0,\alpha]$. We find from Remark \ref{lemma 3.4} (ii) that
\begin{align}\label{4.12}
	\begin{aligned}
		&\quad\left\|(-\Delta)^{\frac{\alpha}{2}}\left[G_{\gamma,\alpha}(t+h-\tau, \cdot)-G_{\gamma,\alpha}(t-\tau, \cdot)\right]\star g(\tau)\right\|_{H^s} \\
		&\leq 2m_{\gamma}(t-\tau)\left[(1 \ast m_{\gamma})(t-\tau)\right]^{\frac{\theta}{\alpha}-1} \sup_{\zeta \in (0,T]}\|g(\zeta)\|_{H^{s+\theta_1}}.
	\end{aligned}
\end{align}
Let $\theta_1 >\alpha$. Taking account of the fact that
\begin{align*}
	&(-\Delta)^{\frac{\alpha}{2}}\left[G_{\gamma,\alpha}(t+h-\tau, \cdot)-G_{\gamma,\alpha}(t-\tau, \cdot)\right]\star g(\tau)\\
	=&\mathcal F^{-1}\Big(|\xi|^\alpha\left[r_\gamma(t+h-\tau,|\xi|^\alpha) - r_\gamma(t-\tau,|\xi|^\alpha) \right] \tilde{g}(\tau,\xi)\Big)
\end{align*}
and Proposition \ref{prop 2.1} (iii), we can bound it by $2 m_{\gamma}(t-\tau)\tilde{g}(t,\xi)(1+|\xi|^2)^{\frac{\theta_1}{2}}$ since $m$ is nonincreasing. Then it follows that
\begin{align*}
	\left\|(-\Delta)^{\frac{\alpha}{2}}\left[G_{\gamma,\alpha}(t+h-\tau, \cdot)-G_{\gamma,\alpha}(t-\tau, \cdot)\right]\star g(\tau)\right\|_{H^s}
	\leq 2m_{\gamma}(t-\tau) \sup_{\zeta \in (0,T]}\|g(\zeta)\|_{H^{s+\theta_1}}.
\end{align*}
Combing this, \eqref{4.12} and strong continuity of  $(-\Delta)^{\frac{\alpha}{2}}G_{\gamma,\alpha}$, we derive
$\lim_{h \to 0^+} I_7(t,h)=0$ for $\theta_1>0$. This ensures the right-continuity of $v_{\alpha,3}$. Thus, the left-continuity of $u_{\alpha,3}$ follows from the similar way.
Recalling \eqref{4.11} we conclude that $\partial_t v_{\alpha,2} \in  C\left((0,T]; H^s \right)$.

As derived in \eqref{4.5} for the case of $ [m_\gamma \ast (-\Delta)^{\frac{\alpha}{2}}v_{\alpha,2}](t)$, one can indeed derive from Proposition \ref{prop 2.2} (ii) that
\begin{align*}
	[m_\gamma \ast (-\Delta)^{\frac{\alpha}{2}}v_{\alpha,2}](t)
	&= (-\Delta)^{\frac{\alpha}{2}}\left(m_\gamma \ast \int_0^\cdot Q_{\gamma,\alpha}(\cdot-\tau)\star g(\tau)\, \mathrm{d}\tau\right)(t)\\
	&= (-\Delta)^{\frac{\alpha}{2}}\left(1 \ast \int_0^\cdot G_{\gamma,\alpha}(\cdot-\tau)\star g(\tau)\, \mathrm{d}\tau\right)(t)\\
	&= \left(1 \ast (-\Delta)^{\frac{\alpha}{2}}\int_0^\cdot G_{\gamma,\alpha}(\cdot-\tau)\star g(\tau)\, \mathrm{d}\tau\right)(t),
\end{align*}
this ensures
\begin{align}\label{4.14}
	\partial_t [m_\gamma \ast (-\Delta)^{\frac{\alpha}{2}}v_{\alpha,2}](t)
	=(-\Delta)^{\frac{\alpha}{2}}\int_0^t G_{\gamma,\alpha}(t-\tau)\star g(\tau)\, \mathrm{d}\tau.
\end{align}
then it also shows $\partial_t [m_\gamma \ast (-\Delta)^{\frac{\alpha}{2}}v_{\alpha,2}] \in C((0,T];H^s)$. Similar to the arguments on \eqref{4.6}, we find $(-\Delta)^{\frac{\alpha}{2}}v_{\alpha,2} \in C\left((0,T]; H^s\right)$. The relation  $\partial_t [m\ast (-\Delta)^{\frac{\alpha}{2}} v_{\alpha,2}] =\frac{1}{\gamma}\partial_t [m_{\gamma}\ast (-\Delta)^{\frac{\alpha}{2}} v_{\alpha,2}]- (-\Delta)^{\frac{\alpha}{2}} v_{\alpha,2}$ leads to  $
\partial_t [m\ast (-\Delta)^{\frac{\alpha}{2}} v_{\alpha,2}] \in  C\left((0,T]; H^s \right)$. Therefore, it follows from \eqref{4.11} and \eqref{4.14} that $v_{\alpha,2}$ is a classical solution to the equation \eqref{4.10}.

\emph{Step 3.}	Given the above, $v_\alpha:=v_{\alpha,1}+v_{\alpha,2}$ is a classical solution to the equation \eqref{neweq 1.2}. Moreover, since the mild solution to the equation \eqref{neweq 1.2} is unique (see \cite[Definition 1.3]{J. Pruss}), and the classical solution given as Definition \ref{def 4.1} is also a mild solution, we conclude that the classical solution $v_\alpha$ is unique. The proof is completed.
\end{proof}

\begin{remark}\label{rem 4.1}
From the proof of Lemma \ref{newprop 2.2} we can find the continuity of $m$ on $\mathbb R_+$ provided $\beta_\gamma>0$. At this moment, we do not assume the continuity condition on $m$ in Theorem \ref{the 4.1}.
\end{remark}

\begin{remark}\label{rem 4.2}
It is well-mentioned that $m\in(\mathcal{H})$ in Theorem \ref{the 4.1} contains a wide range of kernel functions. For instance, we assume that one of the following conditions holds,
\begin{itemize}
	\item[{\rm(i)}] $m \in L^1_{loc}(\mathbb R_+)$ is completely monotonous, i.e., $m \in C^\infty(0,\infty)$ and $(-1)^nm^{(n)}(t) \geq 0$ for $t>0$, $n\in N_+$.
	\item[{\rm(ii)}] $m \in L^1_{loc}(\mathbb R_+)$ is non-negative and non-increasing, moreover, $\ln m$ is convex.
\end{itemize}
Then the conditions of $m$ in Theorem \ref{the 4.1} are satisfied.
For more details, we can refer to \cite[Example 3.3, 3.4, 3.5, 3.6]{F. Alegria}
and \cite{P. Clement 81}.
\end{remark}
\subsection{Convergence}

In this subsection we establish the convergent behavior of classical solutions as the parameter $\alpha$ converges to $2$. The following lemmas present the convergence behavior of $Q_{\gamma,\alpha}(t,\cdot)$ and $G_{\gamma,\alpha}(t,\cdot)$ when $\alpha \to 2^{-}$.
\begin{lemma}\label{lemma 5.1}
Let $\alpha \in (1,2)$, $s >\frac{N}{2}$ and $\beta \in[0,s-\frac{N}{2})$. Assume $m \in (\mathcal{H})$. Then there exist two constants $A_{s, \beta}$ and $A'_{s,\beta}$ such that
\begin{align*}	&\left\|(-\Delta)^{\frac{\beta}{2}}\left[Q_{\gamma,\alpha}(t,\cdot)-Q_{\gamma,2}(t,\cdot)\right]\right\|_{H^{-s}} \leq A_{s,\beta}[t+ \gamma(1\ast m)(t)]^{\varepsilon}(2-\alpha)\\
	&\left\|(-\Delta)^{\frac{\beta}{2}}\left[Q_{\gamma,\alpha}(t,\cdot)-Q_{\gamma,2}(t,\cdot)\right] \right\|_{H^{-s}} \geq
	A'_{s,\beta}\frac{1}{1+[k_\gamma(t)]^{-1}}\frac{(1\ast m_\gamma)(t)}{1+(1\ast m_\gamma)(t)}(2-\alpha)
\end{align*} for $t\in(0,\infty)$.
Especially, we can choose a constant $A'' \in (0, A'_{s,\beta})$ such that
\begin{align*}
	A''(2-\alpha)\leq \sup_{t \in (0,\infty)}\|Q_{\gamma,\alpha}(t,\cdot)-Q_{\gamma,2}(t,\cdot)\|_{H^{-s}} \leq A_{s,\beta}(2-\alpha).
\end{align*}
\begin{proof}
	we apply the mean value theorem (in the variable $\alpha$), and Lemma \ref{lemma 3.7} to obtain
	\begin{align*}
		\left|s_\gamma(t,|\xi|^\alpha)-s_\gamma(t,|\xi|^2)\right| \leq (2-\alpha) \left|\ln |\xi| \right|
	\end{align*}
	for every $|\xi| \neq 0$. Then we proceed to calculate that for $t>0$,
	\begin{align}\label{5.1}
		\begin{aligned} &\quad \left\|(-\Delta)^{\frac{\beta}{2}}\left[Q_{\gamma,\alpha}(t,\cdot)-Q_{\gamma,2}(t,\cdot)\right] \right\|_{H^{-s}}^2\\
			&= \int_{\mathbb R^N}\left(s_\gamma(t,|\xi|^\alpha)-s_\gamma(t,|\xi|^2)\right)^2\frac{|\xi|^{2\beta}}{(1+|\xi|^2)^s}\, \mathrm{d}\xi\\
			&\leq (2-\alpha)^2 \int_{\mathbb R^N} \frac{|\xi|^{2\beta}|\ln |\xi||^2}{(1+|\xi|^2)^s}\, \mathrm{d}\xi\\
			&\leq (2-\alpha)^2 \int_{\mathbb R^N} \frac{\left|\ln |\xi|\right|^2}{(1+|\xi|^2)^{s-\beta}}\, \mathrm{d}\xi\\
			&=A_{s,\beta}^2 (2-\alpha)^2.
		\end{aligned}
	\end{align}
	
	On the other hand, using the mean value theorem again and Lemma \ref{lemma 3.7}, we also have
	\begin{align*}
		\left|s_\gamma(t,|\xi|^\alpha)-s_\gamma(t,|\xi|^2) \right|
		\geq  \frac{1}{4} \frac{1}{1+[k_\gamma(t)]^{-1}}\frac{(1\ast m_\gamma)(t)}{1+(1\ast m_\gamma)(t)} \left| \ln |\xi| \right|, \quad |\xi| \in [\frac{1}{2}, 1].
	\end{align*}
	This implies that
	\begin{align}\label{5.2}
		\begin{aligned}
			&\quad
			\left\|(-\Delta)^{\frac{\beta}{2}}\left[Q_{\gamma,\alpha}(t,\cdot)-Q_{\gamma,2}(t,\cdot)\right] \right\|_{H^{-s}}\\
			&\geq  \frac{1}{4} \frac{1}{1+[k_\gamma(t)]^{-1}}\frac{(1\ast m_\gamma)(t)}{1+(1\ast m_\gamma)(t)} \left(\int_{\frac{1}{2}\leq |\xi| \leq 1} \frac{|\xi|^{2\beta}|\ln |\xi||^2}{(1+|\xi|^2)^s}\, \mathrm{d}\xi \right)^{\frac{1}{2}}(2-\alpha),\\
			&= A'_{s,\beta}\frac{1}{1+[k_\gamma(t)]^{-1}}\frac{(1\ast m_\gamma)(t)}{1+(1\ast m_\gamma)(t)}(2-\alpha)
		\end{aligned}
	\end{align}
	for every $t>0$.
	
	Finally, we take the upper bound on $t \in (0, \infty)$ in \eqref{5.1} and \eqref{5.2}. The proof is completed.
\end{proof}
\end{lemma}

\begin{lemma}\label{lemma 5.2}
Let $\alpha \in (1,2)$, $s >\frac{N}{2}$ and $\beta \in[0,s-\frac{N}{2})$. Assume $m \in (\mathcal{H})$ and $m$ is nonincreasing. Then there exists a constant $B_{s,\beta}$ such that
\begin{align*}
	\left\|(-\Delta)^{\frac{\beta}{2}}\left[G_{\gamma,\alpha}(t,\cdot) - G_{\gamma,2}(t,\cdot)\right] \right\|_{H^{-s}} \leq B_{s,\beta} m_\gamma\left(\frac{t}{2}\right)(2-\alpha)
\end{align*}
for $t\in(0,\infty)$.
\begin{proof}
	It immediately follows from  Lemma \ref{lemma 3.8} (for $\sigma=0$) and the mean value theorem that
	\begin{align*}
		&\quad \left\|(-\Delta)^{\frac{\beta}{2}}\left[G_{\gamma,\alpha}(t,\cdot) - G_{\gamma,2}(t,\cdot)\right] \right\|_{H^{-s}}\\
		&= \left(\int_{\mathbb R^N}\left[r_\gamma(t,|\xi|^\alpha)-r_\gamma(t,|\xi|^2)\right]^2\frac{|\xi|^{2\beta}}{(1+|\xi|^2)^s}\, \mathrm{d}\xi\right)^{\frac{1}{2}}\\
		&\leq (2-\alpha)m_\gamma \left( \frac{t}{2}\right) 4 \left(\int_{\mathbb R^N} \frac{\left|\ln |\xi|\right|^2}{(1+|\xi|^2)^{s-\beta}}\, \mathrm{d}\xi\right)^{\frac{1}{2}}\\
		&=B_{s,\beta}m_\gamma \left( \frac{t}{2}\right)(2-\alpha)
	\end{align*}
	for $t>0$. The proof is completed.
\end{proof}
\end{lemma}

\begin{lemma}\label{lemma 5.3}
Let $\alpha \in (1,2)$ and $s \geq \alpha+\frac{N}{2}$.	Assume $m \in (\mathcal{H})$ and $m$ is nonincreasing. Then there exists a constant $C_{s,\alpha}$ such that
\begin{align*}
	\left\|(-\Delta)^{\frac{\alpha}{2}}G_{\gamma,\alpha}(t,\cdot) - (-\Delta)G_{\gamma,2}(t,\cdot) \right\|_{H^{-s}} \leq C_{s,\alpha} m_\gamma\left(\frac{t}{2}\right)(2-\alpha)
\end{align*}
for $t>0$.
\begin{proof}
	Let $\alpha \in (1,2)$. From the mean value theorem we know that there exists $\varrho \in (1,\alpha)$ such that for arbitrary $\varepsilon \in (0,\varrho]$ satisfying $\varepsilon<s-\frac{N}{2}$, one has
	\begin{align*}
		&\quad \left||\xi|^\alpha r_\gamma(t,|\xi|^\alpha) - |\xi|^2 r_\gamma(t,|\xi|^2)\right|\\
		&= (2-\alpha)\left| r_\gamma(t,|\xi|^\varrho)|\xi|^\varrho \ln|\xi| +  |\xi|^{\varrho}\partial_\varrho r_\gamma(t,|\xi|^\varrho) \right|\\
		&\leq (2-\alpha)\left[r_\gamma(t,|\xi|^\varrho)|\xi|^{\varrho-\epsilon} |\xi|^{\epsilon}\left|\ln|\xi|\right| +  |\xi|^{\varrho-\epsilon}|\partial_\varrho r_\gamma(t,|\xi|^\varrho)| |\xi|^\epsilon \right]\\
		&\leq (2-\alpha) \left(m_\gamma(t) [(1\ast m_\gamma)(t)]^{\frac{\epsilon}{\varrho}-1}|\xi|^{\epsilon}\left|\ln|\xi|\right| +\frac{4\varrho}{\varepsilon} m_\gamma\left(\frac{t}{2} \right) \left[(1\ast m_\gamma)\left(\frac{t}{2}\right)\right]^{\frac{\epsilon}{\varrho}-1}|\xi|^\epsilon \right)\\
		&\leq (2-\alpha)m_\gamma\left(\frac{t}{2} \right) \left[(1\ast m_\gamma)\left(\frac{t}{2}\right)\right]^{\frac{\epsilon}{\varrho}-1}
		\left( |\xi|^{\epsilon}\left|\ln|\xi|\right| +\frac{4\varrho}{\varepsilon}  |\xi|^\epsilon \right)
	\end{align*}
	for $|\xi|\neq 0$, where we have used Lemma \ref{lemma 3.8} and Proposition \ref{prop 2.1} (iv).
	Taking $\varepsilon = \varrho$ in above inequality,	we proceed to calculate
	\begin{align*}
		&\quad \left\|(-\Delta)^{\frac{\alpha}{2}}G_{\gamma,\alpha}(t,\cdot) - (-\Delta)G_{\gamma,2}(t,\cdot) \right\|_{H^{-s}}\\
		&= \left(\int_{\mathbb R^N} \left[|\xi|^\alpha r_\gamma(t,|\xi|^\alpha) -|\xi|^2 r_\gamma(t,|\xi|^2)\right]^2 \frac{1}{(1+|\xi|^2)^s}\, \mathrm{d}\xi \right)^{\frac{1}{2}}\\
		&\leq (2-\alpha)m_\gamma\left(\frac{t}{2} \right)
		\left(\int_{\mathbb R^N} \frac{\left( \left|\ln|\xi|\right| +4 \right)^2|\xi|^{2\varrho}}{(1+|\xi|^2)^s}\, \mathrm{d}\xi \right)^{\frac{1}{2}}\\
		&\leq (2-\alpha)m_\gamma\left(\frac{t}{2} \right)
		\left(\int_{\mathbb R^N} \frac{\left( \left|\ln|\xi|\right| +4 \right)^2}{(1+|\xi|^2)^{s-\alpha}}\, \mathrm{d}\xi \right)^{\frac{1}{2}}\\
		&=C_{s,\alpha} m_\gamma\left(\frac{t}{2} \right)(2-\alpha).
	\end{align*}
	The proof is completed.
\end{proof}
\end{lemma}
Thanks to Theorem \ref{the 4.1}, it is known that the problem \eqref{neweq 1.2} for the case $\alpha\in(0,2)$ and $\alpha=2$ admit a unique classical solution $v_\alpha(t)$ and $v_2(t)$ with the initial $v_{0,\alpha}$ and $v_{0,2}$, respectively. Now we present some convergent behavior of $v_\alpha(t)$ as $\alpha\to 2^-$.
\begin{theorem}\label{the 5.1}
Let the conditions of Theorem \ref{the 4.1} be satisfied and $s> \frac{N}{2}$. We further assume that
\begin{align*}
	\|v_{0,\alpha} - v_{0,2}\|_{H^s} \leq (2-\alpha)^\kappa~~{\it for~ certain}~\kappa>0.
\end{align*}
Then for arbitrary $\beta \in[0,s-\frac{N}{2})$, there exists constants $D_{s}$ and $K_{s,\beta}$ such that
\begin{align*}
	\quad\sup_{0<t\leq T}\left\|(-\Delta)^{\frac{\beta}{2}}\left[v_\alpha(t) - v_2(t)\right] \right\|_{L^\infty}
	\leq K_{s,\beta}(2-\alpha)^\gamma+D_{s,\beta}(1+T)(2-\alpha),
\end{align*}
Moreover, if $s\geq \alpha+\frac{N}{2}$, then
\begin{align*}
	 \left\|\partial_t v_\alpha(t) - \partial_t v_2(t) \right\|_{L^\infty}
	\leq K_{s,\alpha}m_\gamma(t)(2-\alpha)^\gamma+ D'_{s,\beta}\left(m_\gamma\left(\frac{t}{2}\right)+(1\ast m_\gamma)\left(\frac{t}{2}\right) \right)(2-\alpha)
\end{align*}
for $t\in(0, T]$, where $D'_{s,\alpha}$ and $K_{s,\alpha}$ are two constants depending only on $s$ and $\alpha$.
\begin{proof}
	As derived \eqref{4.6}, we have that for $t\in(0,T]$,
	\begin{align*}
		(-\Delta)^{\frac{\beta}{2}}\int_0^t Q_{\gamma,\alpha}(t-\tau,\cdot)\star g(\tau)\, \mathrm{d}\tau = \int_0^t (-\Delta)^{\frac{\beta}{2}}Q_{\gamma,\alpha}(t-\tau,\cdot)\star g(\tau)\, \mathrm{d}\tau.
	\end{align*}
	Recalling the formula of $v_\alpha(t,\cdot)$, we immediately calculate
	\begin{align*}
		&\quad \sup_{0<t\leq T}\left\|(-\Delta)^{\frac{\beta}{2}}\left[v_\alpha(t) - v_2(t)\right] \right\|_{L^\infty}\\
		&\leq \sup_{0<t\leq T} \left\|(-\Delta)^{\frac{\beta}{2}}\left[Q_{\gamma,\alpha}(t,\cdot)\star v_{0,\alpha}- Q_{\gamma,2}(t,\cdot)\star v_{0,2}\right]\right\|_{L^\infty}\\
		&\quad + \sup_{0<t\leq T}\left\|\int_0^t (-\Delta)^{\frac{\beta}{2}} \left[Q_{\gamma,\alpha}(t-\tau, \cdot)- Q_{\gamma,2}(t-\tau,\cdot) \right]\star g(\tau)\, \mathrm{d}\tau \right\|_{L^\infty}\\
		&=: J_{\alpha,1}+J_{\alpha,2}.
	\end{align*}
	Then we continue to consider $J_{\alpha,1}$ and $J_{\alpha,2}$.
	For $J_{\alpha,1}$, it can be divided into two parts:
	\begin{align*}
		J_{\alpha,1} &\leq \sup_{0<t\leq T} \left\|(-\Delta)^{\frac{\beta}{2}}\left[Q_{\gamma,\alpha}(t,\cdot)- Q_{\gamma,2}(t,\cdot)\right]\star v_{0,\alpha}\right\|_{L^\infty}\\
		&\quad+\sup_{0<t\leq T} \left\|(-\Delta)^{\frac{\beta}{2}}Q_{\gamma,2}(t,\cdot)\star (v_{0,\alpha} - v_{0,2})\right\|_{L^\infty}=: J_{\alpha,11}+J_{\alpha,12}.
	\end{align*}
	Taking account of the assumption on $v_{0,\alpha}$, we get $\sup_{1<\alpha \leq 2} \left\|v_{0,\alpha}\right\|_{H^s}<\infty$. Applying Young inequality and Lemma \ref{lemma 5.1}, it follows that
	\begin{align*}
		J_{\alpha,11} &=\sup_{0<t\leq T} \left\|(I-\Delta)^{-\frac{s}{2}}(-\Delta)^{\frac{\beta}{2}}\left[Q_{\gamma,\alpha}(t,\cdot)- Q_{\gamma,2}(t,\cdot)\right]\star (I-\Delta)^{\frac{s}{2}}v_{0,\alpha}\right\|_{L^\infty}\\
		&\leq \sup_{0<t\leq T} \left\|(I-\Delta)^{-\frac{s}{2}}(-\Delta)^{\frac{\beta}{2}}\left[Q_{\gamma,\alpha}(t,\cdot)- Q_{\gamma,2}(t,\cdot)\right]\right\|_{L^2}  \left\| (I-\Delta)^{\frac{s}{2}}v_{0,\alpha}\right\|_{L^2}\\
		&\leq \sup_{0<t\leq T} \left\|(-\Delta)^{\frac{\beta}{2}}\left[Q_{\gamma,\alpha}(t,\cdot)- Q_{\gamma,2}(t,\cdot)\right]\right\|_{H^{-s}} \times \sup_{1<\alpha \leq 2} \left\|v_{0,\alpha}\right\|_{H^s}\\
		&\leq A_{s,\beta}(2-\alpha) \sup_{1<\alpha \leq 2} \left\|v_{0,\alpha}\right\|_{H^s}.
	\end{align*}
	As for $J_{\alpha,12}$, we also arrive at
	\begin{align*}
		J_{\alpha,12}&\leq \sup_{0<t\leq T} \left\|(-\Delta)^{\frac{\beta}{2}}Q_{\gamma,2}(t,\cdot)\right\|_{H^{-s}} \times \sup_{1<\alpha \leq 2} \left\|v_{0,\alpha}-v_{0,\alpha}\right\|_{H^s}\\
		&\leq \left(\int_{\mathbb R^N} \frac{1}{(1+|\xi|^2)^{s-\beta}}\, \mathrm{d}\xi \right)^{\frac{1}{2}}\left\|v_{0,\alpha} - v_{0,2}\right\|_{H^s}\\
		&\leq K_{s,\beta}(2-\alpha)^\gamma.
	\end{align*}
	
	For $J_{\alpha,2}$, using Young inequality and Lemma \ref{lemma 5.1} again, it yields
	\begin{align*}
		J_{\alpha,2}& \leq  \sup_{0<t\leq T}\int_0^t \left\|(I-\Delta)^{-\frac{s}{2}}(-\Delta)^{\frac{\beta}{2}}\left[Q_{\gamma,\alpha}(t-\tau, \cdot)- Q_{\gamma,2}(t-\tau,\cdot) \right]\star (I-\Delta)^{\frac{s}{2}}g(s)\right\|_{L^\infty}\, \mathrm{d}\tau \\
		&\leq \sup_{0<t\leq T}\int_0^t \left\|(-\Delta)^{\frac{\beta}{2}}\left[Q_{\gamma,\alpha}(t-\tau, \cdot)- Q_{\gamma,2}(t-\tau,\cdot)\right] \right\|_{H^{-s}} \left\|g(\tau)\right\|_{H^s}\, \mathrm{d}\tau \\
		&\leq A_{s,\beta}(2-\alpha) T\sup_{0<t\leq T}\left\|g(t)\right\|_{H^{s+\theta_1}}.
	\end{align*}
	Let us choose $D_{s,\beta} =A_{s,\beta}  \left(\sup_{1<\alpha \leq 2} \left\|v_{0,\alpha}\right\|_{H^s}+ \sup_{0<t\leq T}\left\|g(t)\right\|_{H^{s+\theta_1}}\right)$, the first inequality follows.
	
	On the other hand, we apply Lemma \ref{lemma 4.2} (i) and \eqref{4.11} to obtain
	\begin{align*}
		\left\|\partial_t v_\alpha(t) - \partial_t v_2(t) \right\|_{L^\infty}
		&\leq  \left\|\left[(-\Delta)^{\frac{\alpha}{2}} G_{\gamma,\alpha}(t-\tau, \cdot)-(-\Delta) G_{\gamma,2}(t-\tau,\cdot) \right]\star v_{0,\alpha} \right\|_{L^\infty}\\
		&\quad+ \left\|(-\Delta)G_{\gamma,2}(t,\cdot)\star (v_{0,\alpha}-v_{0,2}) \right\|_{L^\infty}\\
		&\quad + \left\|\int_0^t \left[(-\Delta)^{\frac{\alpha}{2}} G_{\gamma,\alpha}(t-\tau, \cdot)-(-\Delta) G_{\gamma,2}(t-\tau,\cdot) \right]\star g(\tau)\, \mathrm{d}\tau \right\|_{L^\infty}\\
		&:=J_{\alpha,3}+J_{\alpha,4}+J_{\alpha,5}.
	\end{align*}
	By the similar way as derived the first inequality and using Lemma \ref{lemma 5.3}, one can obtain the bounds of $J_{\alpha,i}~(i=3,4,5)$. Specifically,
	\begin{align*}
		J_{\alpha,3} \leq C_{s,\alpha} & m_\gamma\left(\frac{t}{2}\right)(2-\alpha)\sup_{1<\alpha \leq 2} \left\|v_{0,\alpha}\right\|_{H^s},~~~~
		J_{\alpha,4} \leq K_{s,\alpha}m_\gamma(t)(2-\alpha)^\gamma,\\
		&J_{\alpha,5} \leq 2C_{s,\alpha}(1\ast m_\gamma)\left(\frac{t}{2}\right)\sup_{0<t\leq T}\left\|g(t)\right\|_{H^{s+\theta_1}}(2-\alpha)
	\end{align*} for $t\in(0,T]$.
	Then we choose $D'_{s,\alpha}= C_{s,\alpha} \sup_{1<\alpha \leq 2} \left\|v_{0,\alpha}\right\|_{H^s} +2C_{s,\alpha}\sup_{0<t\leq T}\left\|g(t)\right\|_{H^{s+\theta_1}}$, the second inequality holds.
\end{proof}
\end{theorem}

\section{Mild solution and approximation of nonlinear problem}

In this section, we consider the nonlinear problem
\begin{align}\label{eq 6.1}
\begin{cases}
	\partial_{t}v_\alpha + (-\Delta)^{\frac{\alpha}{2}}
	v_\alpha + \gamma \partial_{t}(m \ast (-\Delta)^{\frac{\alpha}{2}} v_\alpha) = |v_\alpha|^p, \quad p>1, \ t>0, \ x \in \mathbb R^N, \\
	v_{\alpha}|_{t=0}=v_{0,\alpha}(x), \quad  x \in \mathbb R^N,
\end{cases}
\end{align}
Then we give a sufficient condition of global solutions to the problem \eqref{eq 6.1} and established the convergent behavior of solutions as the parameter $\alpha$ converges to $2$.

Now the concept of mild solutions to the problem \eqref{eq 6.1} is stated.
\begin{definition}\label{def 6.1}
Let $\alpha \in (1,2]$, $s\geq 0$ and $p>1$. Then function $v_\alpha \in C\left([0,\infty); H^s \right)$ is called a global mild solution to problem \eqref{eq 6.1}, if $v_\alpha(t)$ satisfies
\begin{align}\label{6.1}
	v_\alpha(t)= Q_{\gamma,\alpha}(t,\cdot) \star v_{0,\alpha} + \int_0^t Q_{\gamma,\alpha}(t-\tau, \cdot) \star |v_\alpha|^p \mathrm{d}\tau \quad{\it for}~~ (t,x) \in (0, \infty) \times \mathbb R^N.
\end{align}
Let $T>0$. The function $v_\alpha \in C\left([0,T]; H^s \right)$ is called a local mild solution to problem \eqref{eq 6.1}, if $v_\alpha(t)$ satisfies \eqref{6.1} for $(t,x) \in (0, T] \times \mathbb R^N$.
\end{definition}

Let us define the Banach space $E_\alpha$:
$$E_\alpha:=\{v \in C\left(\mathbb R_+; H^s\right): \|v\|_{E_\alpha} < \infty\}$$ with the norm
$
\|v\|_{E_\alpha}:=\sup_{t \in \mathbb R_+} \|v_\alpha(t)\|_{H^s} + \sup_{t \in \mathbb R_+} \left[(1\ast m_\gamma)(t)\right]^{\vartheta}\|v_\alpha(t)\|_{H^{s'+1}}.
$ for $s,\vartheta>0$.
In what follows, we present the main results.	
\begin{theorem}\label{the 6.1}
Assume $m \in (\mathcal{H})$. Let $\alpha \in (1,2]$, $p> \max \left\{\frac{3}{2},\alpha \right\}$ and $\frac{1}{2}<s<p-1$. Take $s_\alpha=s+1-\frac{\alpha-1}{p -1 }$ and $\theta_\alpha = \frac{\alpha(p-1)}{\alpha-1}$. Then for every $v_{0,\alpha} \in H^{s_\alpha}$ satisfying $\|v_{0,\alpha}\|_{H^{s_\alpha}} \leq \varpi$ with a certain $\varpi>0$, the problem \eqref{eq 1.1} admits a unique global mild solution $v_{\alpha}$ satisfying
\begin{align*}
	\sup_{t \in \mathbb R_+} \|v_\alpha(t)\|_{H^{s_\alpha}} + \sup_{t \in \mathbb R_+} \left[(1\ast m_\gamma)(t)\right]^{\frac{1}{\theta_\alpha}}\|v_\alpha(t)\|_{H^{s+1}}< 2(1+2^{\frac{\alpha-1}{\alpha(p-1)}})\varpi.
\end{align*}
If $v'_\alpha(t)$ is the solution corresponding to the initial value $v'_{0,\alpha}$ that satisfies $\|v'_{0,\alpha}\|_{H^{s_\alpha}} \leq \varpi$, then
\begin{align*}
	\|v_\alpha-v'_\alpha\|_{E_\alpha} \leq 2(1+2^{\frac{\alpha-1}{\alpha(p-1)}})\|v_{0,\alpha}-v'_{0,\alpha}\|_{E_\alpha}.
\end{align*}
\begin{proof}
	We use the the contraction mapping theorem to prove our conclusion. Take $\varpi>0$ and $\mu= (1+2^{\frac{\alpha-1}{2(p-\alpha)}})\varpi$. Define the closed ball $\Omega_{\alpha,\mu}:=\{v\in E_\alpha:\|v\|_{E_\alpha}\leq2\mu\}$. Construct the operator $W_\alpha$:
	\begin{align*}
		(W_\alpha v)(t)= Q_{\gamma,\alpha}(t,\cdot) \star v_{0,\alpha} + \int_0^t Q_{\gamma,\alpha}(t-\tau, \cdot) \star |v|^p \mathrm{d}\tau, \quad v \in \Omega_{\alpha,\mu}.
	\end{align*}
	\emph{Step 1.} We prove that $W_\alpha$ maps $\Omega_{\alpha,\mu}$ into itself.
For arbitrary $t\in [0,\infty)$ fixed, we put $0<h<1$ and then consider
	\begin{align*}
		&\quad\int_0^{t+h} Q_{\gamma, \alpha}(t+h-\tau, \cdot) \star |v(\tau)|^p\, \mathrm{d}\tau - \int_0^{t} Q_{\gamma, \alpha}(t-\tau, \cdot) \star |v'(\tau)|^p\, \mathrm{d}\tau\\
		&=\int_0^{t} (Q_{\gamma, \alpha}(t+h-\tau, \cdot) -Q_{\gamma, \alpha}(t-\tau, \cdot)) \star |v(\tau)|^p\, \mathrm{d}\tau + \int_{t}^{t+h} Q_{\gamma, \alpha}(t+h-\tau, \cdot) \star |v'(\tau)|^p\, \mathrm{d}\tau\\
		&= I_9(t,h)+I_{10}(t,h).
	\end{align*}
	From the assumptions of $s_\alpha$ and $s$, it can be inferred that $s_\alpha>s>\frac{1}{2}$. Then \cite[Lemma 2.6]{R. P. De Moura} shows that there exists a constant $L_0>0$ such that $\||v|^p \|_{H^{s_\alpha}} \leq L_0 \|v\|_{H^{s_\alpha}}^{p}$. This along with Lemma \ref{3.1} (i), Lemma \ref{lemma 4.1} (i) and the Lebesgue's dominated convergence theorem gives
	\begin{align*}
		&0\leq I_9(t,h) \leq L_0\|v\|^p_{H^{s_\alpha}}\int_0^t
		\ln \frac{(1\ast m_\gamma)(t+h-\tau)}{(1\ast m_\gamma)(t-\tau)}\, \mathrm{d}\tau \to 0 \text{ ~as~ } h \to 0^+,\\
		&0\leq I_{10}(t,h) \leq L_0\|v\|^p_{H^{s_\alpha}}h \to 0 \text{ ~as~ } h \to 0^+,
	\end{align*}
	which ensures the right-continuity of $\int_{0}^{t} Q_{\gamma, \alpha}(t-\tau, \cdot) \star |v'(\tau)|^p\, \mathrm{d}\tau$ on $\mathbb R_+$. Similarly, left continuity can also be derived. Furthermore, the Lemma \ref{lemma 4.1} (i) gives that $Q_{\gamma, \alpha}(t, \cdot) \star v_{0,\alpha} \in C\left(\mathbb R_+; H^{s_\alpha}\right)$. From the above we obtain $W_\alpha v \in C\left(\mathbb R_+; H^{s_\alpha}\right)$ for $v \in E_{\alpha}$.
	
	Due to Lemma \ref{lemma 3.1} (i), we get
	\begin{align*}
		&\sup_{t \in \mathbb R_+}\left\|Q_{\gamma, \alpha}(t, \cdot) \star v_{0,\alpha} \right\|_{H^{s_\alpha}} \leq \|v_{0,\alpha}\|_{H^{s_\alpha}},\\
		&\sup_{t \in \mathbb R_+}\left[(1 \ast m_\gamma)(t) \right]^{\frac{1}{\theta_\alpha}}\left\|Q_{\gamma, \alpha}(t, \cdot) \star v_{0,\alpha} \right\|_{H^{s+1}} \leq 2^{\frac{s+1-s_\alpha}{2}}\|v_{0,\alpha}\|_{H^{s_\alpha}}.
	\end{align*}
	For any $v_{0,\alpha} \in H^{s_\alpha}$ satisfying $\|v_{0,\alpha}\|_{H^{s_\alpha}}<\varpi$, it holds that
	\begin{align}\label{6.3}
		\left\|Q_{\gamma, \alpha}(t, \cdot) \star v_{0,\alpha} \right\|_{E_\alpha} \leq \mu.
	\end{align}
Since $\frac{s_\alpha-s}{\alpha}+\frac{p}{\theta_\alpha} = 1$, it follows from \cite[Lemma 2.7]{R. P. De Moura}, Lemma \ref{3.1} (i) and the choice of $s_\alpha$ that there exists a constant $L_1>0$ such that
	\begin{align*}
		\begin{aligned}
			&\quad\left\|\int_0^t Q_{\gamma, \alpha}(t-\tau, \cdot) \star |v(\tau)|^p)\, \mathrm{d}\tau \right\|_{H^{s_\alpha}}\\ &\leq \int_0^t \left\|Q_{\gamma,\alpha}(t-\tau, \cdot)\star |v(\tau)|^p \right\|_{H^{s_\alpha}} \, \mathrm{d}\tau\\
			&\leq  2^{\frac{s_\alpha-s}{2}}\int_0^t \left[(1 \ast m_\gamma)(t-\tau) \right]^{-\frac{s_\alpha-s}{\alpha}}\left\| |v(\tau)|^p \right\|_{H^{s}} \, \mathrm{d}\tau\\
			&\leq  2^{\frac{s_\alpha-s}{2}}L_1\int_0^t \left[(1 \ast m_\gamma)(t-\tau) \right]^{-\frac{s_\alpha-s}{\alpha}}\|v(\tau)\|^p_{H^{s+1}} \, \mathrm{d}\tau\\
			&\leq  2^{\frac{s_\alpha-s}{2}}L_1 \|v\|^p_{E_\alpha}\int_0^t \left[(1 \ast m_\gamma)(t-\tau) \right]^{-\frac{s_\alpha-s}{\alpha}}\left[(1\ast m_\gamma)(\tau)\right]^{-\frac{p}{\theta_\alpha}} \, \mathrm{d}\tau \quad \text{ ~for~ } t\geq 0.
		\end{aligned}
	\end{align*}
	Recall that $\frac{s_\alpha-s}{\alpha}+\frac{p}{\theta_\alpha}=1$. We proceed to estimate the integral
	\begin{align*}
		&\quad \int_0^t \left[(1 \ast m_\gamma)(t-\tau) \right]^{-\frac{s_\alpha-s}{\alpha}}\left[(1\ast m_\gamma)(\tau)\right]^{-\frac{p}{\theta_\alpha}} \, \mathrm{d}\tau\\
		&\leq \int_0^{\frac{t}{2}} m_\gamma(\tau)\left[(1 \ast m_\gamma)(t-\tau) \right]^{-\frac{s_\alpha-s}{\alpha}}\left[(1\ast m_\gamma)(\tau)\right]^{-\frac{p}{\theta_\alpha}} \, \mathrm{d}\tau\\
		&\quad+\int_{\frac{t}{2}}^t m_\gamma(t-\tau)\left[(1 \ast m_\gamma)(t-\tau) \right]^{-\frac{s_\alpha-s}{\alpha}}\left[(1\ast m_\gamma)(\tau)\right]^{-\frac{p}{\theta_\alpha}} \, \mathrm{d}\tau\\
		&\leq \frac{1}{1-\frac{s_\alpha-s}{\alpha}}+\frac{1}{1-\frac{p}{\theta_\alpha}} = \frac{\alpha^2(p-1)^2}{p(p-\alpha)(\alpha-1)}.
	\end{align*}
It immediately yields that
	\begin{align}\label{6.4}
		\begin{aligned}
			\left\|\int_0^t Q_{\gamma, \alpha}(t-\tau, \cdot) \star |v(\tau)|^p)\, \mathrm{d}\tau \right\|_{H^{s_\alpha}}\leq 2^{\frac{s_\alpha-s}{2}} \frac{\alpha^2(p-1)^2}{p(p-\alpha)(\alpha-1)}L_1 \|v\|^p_{E_\alpha}
		\end{aligned}
	\end{align}
for $t\geq 0$.
	Similarly, it can be inferred that
	\begin{align*}
		&\quad \left\|\int_0^t Q_{\gamma, \alpha}(t-\tau, \cdot) \star |v(\tau)|^p\, \mathrm{d}\tau \right\|_{H^{s+1}}\\
		&\leq \int_0^t \left\|Q_{\gamma,\alpha}(t-\tau, \cdot)\star |v(\tau)|^p \right\|_{H^{s+1}} \, \mathrm{d}\tau\\
		&\leq  \sqrt{2}\int_0^t \left[(1 \ast m_\gamma)(t-\tau) \right]^{-\frac{1}{\alpha}}\left\||v(\tau)|^p \right\|_{H^{s}} \, \mathrm{d}\tau\\
		&\leq  \sqrt{2}L_1 \|v\|^p_{E_\alpha}\int_0^t \left[(1 \ast m_\gamma)(t-\tau) \right]^{-\frac{1}{\alpha}}[(1\ast m_\gamma)(t)]^{-\frac{p}{\theta_\alpha}} \, \mathrm{d}\tau\\
		&= \sqrt{2}\frac{\alpha(p\alpha -2\alpha+1)}{(p-\alpha)(\alpha-1)}L_1  [(1\ast m_\gamma)(t)]^{1-\frac{1}{\alpha} - \frac{p}{\theta_\alpha}}\|v\|^p_{E_\alpha}
	\end{align*}
for $t>0$. The choice of $\theta$ gives $\frac{1}{\alpha}+\frac{p}{\theta_\alpha} = 1+\frac{1}{\theta_\alpha}$. This implies
	\begin{align}\label{6.5}
		\begin{aligned}
			&\quad \sup_{t \in \mathbb R_+} \left[(1 \ast m_\gamma)(t)\right]^{\frac{1}{\theta_\alpha}} \left\|\int_0^t Q_{\gamma, \alpha}(t-\tau, \cdot) \star |v(\tau)|^p\, \mathrm{d}\tau \right\|_{H^{s+1}}\\
			&\leq \sqrt{2}\frac{\alpha(p\alpha -2\alpha+1)}{(p-\alpha)(\alpha-1)}L_1 \|v\|^p_{E_\alpha}.
		\end{aligned}
	\end{align}
	Take
	$$\varpi\leq \frac{1}{1+2^{\frac{\alpha-1}{\alpha(p-1)}}}\left[2^{2p-1} \left(2^{\frac{s_\alpha-s}{2}}\frac{\alpha^2(p-1)^2}{p(p-\alpha)(\alpha-1)} + \sqrt{2}\frac{\alpha(p\alpha -2\alpha+1)}{(p-\alpha)(\alpha-1)} \right)L_1  \right]^{-\frac{1}{p-1}}.$$
	It follows from \eqref{6.3}, \eqref{6.4} and \eqref{6.5} that $\|W_\alpha v\|_{E_\alpha} \leq 2\mu \text{ ~for~ } v\in \Omega_{\alpha,\mu}$. Immediately, $W_\alpha$ maps $\Omega_{\alpha,\mu}$ into itself.
	
	\emph{Step 2.} We prove that $W_\alpha$ is a contraction operator. Using \cite[Lemma 2.7]{R. P. De Moura} and Lemma \ref{3.1} (i) again, it holds that for any $v, v' \in \Omega_{\alpha,\mu}$,
	\begin{align*}
		&\quad \|W_\alpha v - W_\alpha v'\|_{H^{s_\alpha}}\\
		&\leq \int_0^t \left\|Q_{\gamma, \alpha}(t-\tau, \cdot) \star \left(|v(\tau)|^p-|v'(\tau)|^p\right)\right\|_{H^{s_\alpha}}\, \mathrm{d}\tau \\
		&\leq  2^{\frac{s_\alpha-s}{2}}\int_0^t \left[(1 \ast m_\gamma)(t-\tau) \right]^{-\frac{s_\alpha-s}{\alpha}}\left\||v(\tau)|^p - |v'(\tau)|^p \right\|_{H^{s}} \, \mathrm{d}\tau\\
		&\leq  2^{\frac{s-\alpha-s}{2}}L_1\int_0^t \left[(1 \ast m_\gamma)(t-\tau) \right]^{-\frac{s_\alpha-s}{\alpha}}\|v(\tau)-v'(\tau)\|_{H^{s+1}}\left(\|v(\tau)\|_{H^{s+1}} + \|v'(\tau)\|_{H^{s+1}}\right)^{p-1} \, \mathrm{d}\tau\\
		&\leq  2^{\frac{s_\alpha-s}{2}}L_1 \|v-v'\|_{E_\alpha}\left(\|v\|_{E_\alpha} +\|v'\|_{E_\alpha} \right)^{p-1}\int_0^t \left[(1 \ast m_\gamma)(t-\tau) \right]^{-\frac{s_\alpha-s}{\alpha}}\left[(1\ast m_\gamma)(\tau)\right]^{-\frac{p}{\theta_\alpha}} \, \mathrm{d}\tau\\
		&= 2^{\frac{s_\alpha-s}{2}}L_1 \frac{\alpha^2(p-1)^2}{p(p-\alpha)(\alpha-1)}(4\mu)^{p-1}\|v-v'\|_{E_\alpha}.
	\end{align*}
	With similar arguments, we get
	\begin{align*}
		\sup_{t \in \mathbb R_+} \left[(1\ast m_\gamma)(t) \right]^{\frac{1}{\theta_\alpha}}\| W_\alpha v -W_\alpha v'\|_{H^{s+1}}<\sqrt{2}\frac{\alpha(p\alpha -2\alpha+1)}{(p-\alpha)(\alpha-1)}L_1 (4\mu)^{p-1} \|v-v'\|_{E_\alpha}.
	\end{align*}
	Recall the choice of $\mu$, we have
	\begin{align}\label{6.6}
		\|W_\alpha v-W_\alpha v'\|_{E_\alpha} \leq \frac{1}{2} \|v-v'\|_{E_\alpha}.
	\end{align}
	
	According to the contraction mapping theorem, the operator $W_\alpha$ admits a unique fixed point $v_\alpha$ in $\Omega_{\alpha,\mu}$, which is a global mild solution to the problem \eqref{eq 6.1}. If $v'_\alpha$ is the solution corresponding to the initial value $v'_{0,\alpha}$. As the similar arguments we derived \eqref{6.6}, it yields
	$$
	\|v_\alpha-v'_\alpha\|_{E_\alpha} \leq (1+2^{\frac{\alpha-1}{\alpha(p-1)}})\|v_0-v'_0\|_{H^{s_\alpha}}+\frac{1}{2}\|v-v'\|_{E_\alpha}.
	$$
	Then the continuous dependence follows. The proof is completed.
\end{proof}
\end{theorem}

\begin{theorem}\label{the 6.2}
Assume $m \in (\mathcal{H})$. Let $p>\max\{2, 1+\frac{N}{2}\}$, $s\in(\frac{N}{2},p-1)$ and $\alpha \in [1+\varepsilon_1,2]$ for $\varepsilon_1 \in (0,1)$. Take $s_\alpha=s+1-\frac{\alpha-1}{p -1 }$, $\theta_\alpha = \frac{\alpha(p-1)}{\alpha-1}$. If we can choose a certain $\varpi_T>0$ related to $T>0$ such that for any $v_{0,\alpha}$ satisfying $\|v_{0,\alpha}\|_{H^s}\leq \varpi_T$, we have
\begin{align*}
	\|v_{0,\alpha} - v_{0,2}\|_{H^{s_2}} \leq (2-\alpha)^\kappa~~{\it for~a~ certain}~\kappa>0,
\end{align*}
then there exist two constants $M_1$ and $M_2$ that depends only on $s, s, p, \varepsilon_1$, such that the problem \eqref{eq 1.1} admits
a global mild solution $v_\alpha$ corresponding to the initial value $v_{0,\alpha}$ satisfying
\begin{align*}
	\sup_{t \in [0,	T]} \|v_\alpha(t) - v_2(t)\|_{L^\infty} \leq M_1\left(1+[(1\ast m_\gamma)(T)]^{\frac{p-2}{2(p-1)}}\right)(2-\alpha) +M_2(2-\alpha)^\kappa.
\end{align*}
\begin{proof}
	Let $\varpi'\leq \left[2^{2p-1} \left(2^{\frac{s_\alpha-s}{2}}\frac{4(p-1)^2}{p(p-2)\varepsilon_1} + \sqrt{2}\frac{2(2p -3)}{(p-2)\varepsilon_1} \right)L_1  \right]^{-\frac{1}{p-1}}$, $\mu'=2\left(1+2^{\frac{1}{{2(p-1)}}}\right)\varpi'$ and
	\begin{align*}
		\varpi_T=\frac{\varpi'}{\max \left \{1, [(1\ast m_\gamma)(T)]^{\frac{2-\alpha}{2\alpha(p-1)}}\right\}},\ \mu_T=2\left(1+2^{\frac{1}{{2(p-1)}}}\right)\varpi_T.
	\end{align*}
	It follows from Theorem \ref{the 6.1} that the problem \eqref{eq 1.1} has a global mild solution $v_\alpha$ in $\Omega_{\alpha, \mu_T}$ for $\|v_{0,\alpha}\|_{H^{s_\alpha}}\leq \varpi_T$ with $\alpha \in [1+\varepsilon_1,2]$.
	
	Define the Banach space $E^T_{\alpha}:=\{v \in C\left([0,T]; H^{s_\alpha}\right): \|v\|_{E_{\alpha,T}} < \infty\}$ and the close ball $\Omega^T_{\alpha, \mu_T}=\{v \in E^T_{\alpha}:\|v\|_{E^T_{\alpha}} \leq 2\mu_T\}$, where
	\begin{align*}
		&\|v\|_{E^T_{\alpha}}:=\sup_{t \in [0,T]} \|v(t)\|_{H^{s_\alpha}} + \sup_{t \in [0,T]} \left[(1\ast m_\gamma)(t)\right]^{\frac{1}{\theta_\alpha}}\|v(t)\|_{H^{s+1}}.
	\end{align*}
	Notice $H^{s_\alpha} \hookrightarrow H^{s_2}$ for $\theta_\alpha >\theta_2$. This along with the choice of $\varpi_T$ gives
	$$v_\alpha \in C\left([0,T]; H^{s_2} \right)~~{\rm and}~ \|v_\alpha\|_{E^T_2} \leq \|v_\alpha\|_{E^T_\alpha} \leq \mu'.$$
	Immediately, $v_\alpha(t)$($t \in [0,T]$) is the unique local mild solution in $\Omega^T_{2,\mu'}$. As we derived the continuous dependency in Theorem \ref{the 6.1}, it is sufficient to prove
	\begin{align}\label{6.7}
		\|v_\alpha -v_2\|_{E_{2,T}}\leq 2(1+2^{\frac{1}{2(p-1)}}) \|v_{0,\alpha}-v_{0,2}\|_{H^{s_2}}.
	\end{align}
Indeed, we have
	\begin{align*}
		&\quad \sup_{t\in [0,T]}\left\|v_\alpha(t) - v_2(t)\right\|_{L^\infty}\\
		&\leq \sup_{t\in (0,T]} \left\|Q_{\gamma,\alpha}(t,\cdot)\star v_{0,\alpha}- Q_{\gamma,2}(t,\cdot)\star v_{0,2}\right\|_{L^\infty}\\
		&\quad + \sup_{t\in  (0,T]}\left\|\int_0^t \left(Q_{\gamma,\alpha}(t-\tau, \cdot)- Q_{\gamma,2}(t-\tau,\cdot) \right)\star |v_\alpha(\tau)|^p\, \mathrm{d}\tau \right\|_{L^\infty}\\
		&\quad + \sup_{t\in (0,T]}\left\|\int_0^t Q_{\gamma,2}(t-\tau,\cdot)\star \left(|v_\alpha(\tau)|^p - |v_2(\tau)|^p\right)\, \mathrm{d}\tau \right\|_{L^\infty}\\
		&:= J'_{\alpha,1}+J'_{\alpha,2} + J'_{\alpha,3}.
	\end{align*}
	
We consider $J'_{\alpha,1}$. Similar to the proof of Theorem \ref{the 5.1}, from the assumption of $v_{0,\alpha}$, Young' inequality, and Lemma \ref{lemma 5.1} (taking $\beta=0$) we arrive at
	\begin{align*}
		J'_{\alpha,1} \leq A_{s,0} \sup_{1+\varepsilon_1\leq\alpha \leq 2} \left\|v_{0,\alpha}\right\|_{H^{s_2}}(2-\alpha) + K_{s,0}(2-\alpha)^\kappa.
	\end{align*}
For $J'_{\alpha,2}$, we apply Young' inequality, Lemma \ref{lemma 5.1} along with \cite[Lemma 2.6]{R. P. De Moura} and $H^{s'}\hookrightarrow H^s$ once again to obtain
	\begin{align*}
		J'_{\alpha,2}& \leq  \sup_{t\in (0,T]}\int_0^t \left\|(I-\Delta)^{-\frac{s}{2}}\left[Q_{\gamma,\alpha}(t-\tau, \cdot)- Q_{\gamma,2}(t-\tau,\cdot) \right]\star (I-\Delta)^{\frac{s}{2}}|v_\alpha(\tau)|^p\right\|_{L^\infty}\, \mathrm{d}\tau \\
		&\leq \sup_{t\in (0,T]}\int_0^t \left\|Q_{\gamma,\alpha}(t-\tau, \cdot)- Q_{\gamma,2}(t-\tau,\cdot) \right\|_{H^{-s}} \left\||v_\alpha(\tau)|^p\right\|_{H^{s}}\, \mathrm{d}\tau \\
		&\leq A_{s,0}L_1\|v_\alpha\|^p_{E_\alpha}(2-\alpha)\sup_{t\in (0,T]} \int_0^t  [(1 \ast m_\gamma)(\tau)]^{-\frac{p}{2(p-1)}} \, \mathrm{d}\tau \\
		&\leq \frac{2(p-1)}{(p-2)}A_{s,0}L_1(2\mu_1)^p(2-\alpha) [(1 \ast m_\gamma)(T)]^{\frac{p-2}{2(p-1)}}.
	\end{align*}
	For $J'_{\alpha,3}$, from \eqref{6.7} we obtain
	\begin{align*}
		J'_{\alpha,3}& \leq  \sup_{0<t\leq T} \left\|\int_0^t Q_{\gamma,2}(t-\tau,\cdot)\star \left(|v_\alpha(\tau)|^p - |v_2(\tau)|^p\right)\, \mathrm{d}\tau \right\|_{H^{s_2}}\\
		& \leq \|v_\alpha -v_2\|_{E_2^T} \leq 2(1+2^{\frac{1}{2(p-1)}})(2-\alpha)^\kappa.
	\end{align*}
	Let $M_1:= A_{s',0} \sup_{1+\varepsilon_1\leq\alpha \leq 2} \left\|v_{0,\alpha}\right\|_{H^s} + \frac{2(p-1)}{(p-2)}A_{s',0}L_1(2\mu_1)^p,\ M_2:= K_{s,0}+ 2(1+2^{\frac{1}{2(p-1)}})$. The proof is completed.
\end{proof}
\end{theorem}

\noindent{\bf Acknowledgements}\\
This work was supported by National Natural Science Foundation of China (12471172).\\
\noindent{\bf Disclosure statement}\\
No potential conflict of interest was reported by the author.

\end{document}